\documentclass[11pt]{article}
\usepackage{amsmath,amsthm,amssymb}
\usepackage[toc,page]{appendix} 
\usepackage{mathtools}
\usepackage{stmaryrd}
\usepackage[colorlinks]{hyperref}
\usepackage[T1]{fontenc}
\usepackage[twoside]{geometry}
\usepackage{graphics}
\usepackage{booktabs} 
\newtheorem{theorem}{Theorem}[section]

\newtheorem{lemma}[theorem]{Lemma}
\newtheorem{corollary}[theorem]{Corollary}

\newtheorem{remark}[theorem]{Remark} 
 
\numberwithin{equation}{section}

\usepackage{amssymb, amscd, mathrsfs, wasysym,bbm} 
\usepackage{wasysym}

\newcommand \bei {\begin{itemize}}
\newcommand \eei {\end{itemize}}
\newcommand \be {\begin{equation}}
\newcommand \bel {\begin{equation}\label}
\newcommand \ee {\end{equation}}
\newcommand \del \partial
\newcommand \eps \epsilon

\newcommand \SL {\mathcal S}
\newcommand \SLO {\mathcal V}
\newcommand \DL {\mathcal D}
\newcommand \DLO {\mathcal K}
\newcommand \jump[1]{\llbracket #1 \rrbracket}

\newcommand{\Hh}{H^{\frac12}}
\newcommand{\Hnh}{H^{-\frac12}}

\newcommand{\Vn}[1]{V_0(#1)^3}
\newcommand{\Vns}[1]{V_0(#1)}

\newcommand{\id}{{\normalfont\hbox{1\kern-0.15em \vrule width.8pt depth-.5pt}}}
\newcommand{\ra}[1]{\renewcommand{\arraystretch}{#1}}

\begin{document}

\title{\bf \Large Boundary integral equations for  isotropic linear elasticity} 
\author{Benjamin Stamm$^1$, Shuyang Xiang$^1$
\\
{\footnotesize $^1$ MATHCCES, Department of Mathematics, RWTH Aachen University, Schinkelstrasse 2, D-52062 Aachen, Germany}
}
\maketitle
\begin{abstract}
This articles first investigates boundary integral operators for the three-dimensional isotropic linear elasticity of a biphasic model with piecewise constant Lam\'e coefficients in the form of a bounded domain of arbitrary shape surrounded by a background material.
In the simple case of a spherical inclusion, the vector spherical harmonics consist of eigenfunctions of the single and double layer boundary operators and we provide their spectra. 
Further, in the case of many spherical inclusions with isotropic materials, each with its own set of Lam\'e parameters, we propose an integral equation and a subsequent Galerkin discretization using the vector spherical harmonics and apply the discretization to several numerical test cases. 
\end{abstract}

\section{Introduction}
 We consider  three-dimensional boundary value or interface problems of the isotropic elasticity equation related to the following operator: 
\bel{elasticity}
{\bf L} u  :=  -  \mathop{\rm div} \Big(2\mu e(u)+  \lambda\mathop{\rm Tr} \big(e( u)\big) \mathop{\rm Id } \Big), 
\ee
where 
the strain tensor reads $e(u) = {\frac12}(\nabla u+\nabla u^\top)$. It is obvious to see that the operator ${\bf L}$ is self-adjoint on $L^2(\mathbb R^3)^3$. 

In the definition of the operator \eqref{elasticity},  $\mu, \lambda\in  \mathbb R, \mu>0, 2\mu+3\lambda>0$ are the so-called (constant) Lam\'e parameters. The parameter $\mu$ denotes the shear modulus which describes the tendency of the object to deform at a constant volume when being imposed with opposing forces. The other Lam\'e parameter $\lambda$ has no physical meanings but is introduced to simplify the definition of the operator \eqref{elasticity}.  Indeed, it is related to the bulk modulus $K$ through the relation 
\[
\lambda = K - \frac{2}{3} \mu,
\]
where  the bulk modulus $K$  represents the object's tendency to deform in all directions when acted on by opposing force from all directions. We refer to~\cite{P-S} for more detailed descriptions of the Lam\'e parameters. It is sometimes useful to introduce Poisson's ratio $\nu$ which is defined by 
\bel{Poisson-ratio}
\nu = {\lambda\over 2(\mu+\lambda)},
\ee
and whose admissible range is $(-1,1/2)$. The material is extremely compressible in the limit $\nu\to-1$ while extremely incompressible in the other limit $\nu\to 1/2$~\cite{PHCM}.

A model of linear elasticity with appropriate boundary conditions can be approximated by the classic finite element method, see for example~\cite{Leroy, falk} just to name a few contributions from an abundant body of literature, for the general case with non-homogeneous source term. 
On the other hand, displacement fields $u$ being homogeneous solutions, i.e., ${\bf L} u =0$ within a given domain,  can also be represented by isotropic elastic potentials~\cite{BUI,KUPR} and elasticity in piecewise constant isotropic media can then be treated as integral equations for specified interface conditions. 
At the origin of the integral formulation lies the definitions of layer potentials and their corresponding integral operators~\cite{BUI, sauter2010boundary} based on the Green's function~\cite{allan-bower} in the context of the isotropic linear elasticity. 

In particular, on a unit sphere, one can introduce the vector spherical harmonics forming an orthonormal basis of $[L^2(\mathbb S^2)]^3$ and which are eigenfunctions of the corresponding double and single layer boundary operators based on the Green's function~\cite{allan-bower} of isotropic linear elasticity.
The vector spherical harmonics were introduced in~\cite{Hill,weinberg} as an extension of the scalar spherical harmonics~\cite{MacRobert,hobson} to the vectorial case. 
They were further used in the discretization of different physical models such as the Navier-Stokes equations~\cite{Corona-Veera} or Maxwell's equations~\cite{EELL, BEG}. 
However, they are not widely used and only sparely reported in literature, in particular in the context of isotropic elasticity. We demonstrate in this article that the corresponding integral operators have interesting spectral properties which can be made explicit by employing the vector spherical harmonics.

Our main motivation for this work is the derivation of an integral equation to model elastic materials represented by piecewise constant Lam\'e constants with spherical inclusions following similar principles that were presented in \cite{ManyBodyPolTheory,PART1,PART2} in the case of scalar diffusion.
The particular choice of the vector spherical harmonics as basis functions for a Galerkin discretization thereof leads then to an efficient and stable numerical scheme by exploiting the spectral properties of the involved integral operators.
A similar physical model was introduced in ~\cite{Urklain} with an algebraic formula of the approximate solution. However, with the spectral properties of the layer potentials and integral operators at hand, our approach first introduces an integral formulation for the exact solution and thus a rigorous mathematical framework.
In a second step, we then propose the Galerkin discretization.
The mathematical framework lays out the basis to derive a rigorous error analysis which we plan in the future.

We summarize the main contributions and organization of this work as follows:
\bei
\item In Section~\ref{sec:preliminaries} and~\ref{sec:layer-potentials}, we give an introduction and overview of the layer potentials and corresponding boundary integral operators of the isotropic linear elasticity operator \eqref{elasticity} on an arbitrary bounded domain with Lipschitz boundary which are sparely reported in the literature. 
\item 
Analytical properties of layer potentials and boundary integral operators are presented and proven in Section~\ref{sec:properties}. 
\item On the unit sphere, we introduce the vector spherical harmonics in Section~\ref{sec:VSH} and prove spectral properties of the boundary operators and layer potentials of this particular basis. 
\item As an application, we consider a piecewise constant elastic model with spherical inclusions and derive a integral equation in Section~\ref{sec:case-study} that is then discretized by means of the vector spherical harmonics and tested numerically in Section~\ref{sec:8}. 
\eei

%

\section{Preliminaries}
Denote $\mathbb S^2$  the unit sphere and $B $ the unit ball in $\mathbb R^3$. 
Let throughout this paper $\Omega^-\subset \mathbb R^3$ denote a bounded domain with Lipschitz boundary $\Gamma=\del \Omega^-$ and outward pointing normal vector field ${\bf n}: \Gamma\to \mathbb S^2$. Further, we 
denote by $\Omega^+$ the unbounded set  $\mathbb R^3\backslash \overline {\Omega^-}$.

\label{sec:preliminaries}
\subsection{Notations}
\label{ssec:not}
We will first introduce some standard notions in the context of integral equations which can be found in standard textbooks (see, for example, \cite{McLean,sauter2010boundary,steinbach2007numerical}).

Let $\Omega$ be a domain with Lipschitz boundary, e.g., $\Omega=\Omega^-$ or $\Omega=\Omega^+$ (unbounded).
Following the conventions and notation of~\cite{sauter2010boundary}, we define for $s\in\mathbb R$
\begin{align}
	\label{s-loc}
	H_{\mathop{\rm loc}}^s (\Omega)
	&= 
	\left\{ u\in\big(C_{\rm comp}^\infty(\Omega)\big)^* \;\middle|\; \forall \chi \in C_{\rm comp}^\infty(\Omega): \chi u \in H^\ell(\Omega) \right\},
\end{align}
see Definition 2.6.1 in~\cite{sauter2010boundary}, and note that this consist of a slightly unconventional definition of  $H_{\mathop{\rm loc}}^\ell(\Omega)$, see also Remark 2.6.2. 
We further define, see Definition 2.6.5 in~\cite{sauter2010boundary}, for $s\in\mathbb R$
\bel{s-comp}
	H_{\mathop{\rm comp}}^{s} (\Omega) = \bigcup\limits_K \left\{u \in H_{\mathop{\rm loc}}^{s}(\Omega)\;\middle|\; \mathop{\rm supp} (u)\subset K \right\},
\ee
where the union is taken over all relatively compact subsets $K\subset \Omega$, and introduce
\be
\Vns{\Omega^-} = \left\{ v\in H^1(\Omega^-) \;\middle|\; \int_{\Omega^-} v = 0\right\}.
\ee

\noindent
Next, we denote by $\Hh(\Gamma)^3$ the Sobolev space of order $\frac{1}{2}$ with the usual Sobolev-Slobodeckij norm 
$\Vert \lambda \Vert^2_{\Hh(\Gamma)} 
:= 
\sum_{k=1}^3 \Vert \lambda_k \Vert^2_{\Hh(\Gamma)}$
for $\lambda = (\lambda_1,\lambda_2,\lambda_3)^\top$ and with
\[
\Vert \lambda_k \Vert^2_{\Hh(\Gamma)}
:=\Vert \lambda_k \Vert^2_{L^2(\Gamma)} 
+ \int_{\Gamma} \int_{\Gamma} \frac{\vert\lambda_k(x)-\lambda_k(y)\vert^2}{\vert x - y \vert^3 } \, dx dy.
\]
Moreover, we define $\Hnh(\Gamma)^3:=\left(\Hh(\Gamma)^3\right)^*$ and we equip this Sobolev space with the canonical dual norm $\Vert \cdot \Vert_{\Hnh(\Gamma)}$. 
We introduce 
\bel{trace}
\gamma^\mp \colon H_{\mathop{\rm loc}}^1(\Omega^{\mp})^3 \rightarrow \Hh(\Gamma)^3
\ee
as the continuous, linear and surjective interior and exterior Dirichlet trace operators respectively, 
see Theorem 2.6.8~\cite{sauter2010boundary}, and define the jump operator by 
\bel{jump-operator}
\jump{\varphi} = \gamma^-\varphi - \gamma^+\varphi.
\ee
Further, let $\gamma \colon H_{\mathop{\rm loc}}^1(\mathbb R^3)^3 \rightarrow \Hh(\Gamma)^3$ be given by $\gamma\varphi=\gamma^-\varphi =\gamma^+\varphi$ almost everywhere.

Consider now the stress tensor $\mathcal T$ associate with ${\bf L}$, as is  defined by \eqref{elasticity}, reading 
\bel{stress}
\mathcal T\varphi:= 2\mu e(\varphi)+ \lambda \mathop{\rm Tr} e(\varphi )  \mathop{\rm Id}, \quad \varphi \in H^1_{\mathop {\rm loc}}(\Omega^\mp )^3. 
\ee
For the domains $\Omega^\mp$, 
the classical normal derivative operator, satisfying 
\bel{conormal-strong}
	\mathcal T^\mp_{\bf n} \varphi:= \gamma^\mp (\mathcal T \varphi  {\bf n}),
\ee
for regular $\varphi$, can be extended to an operator $\mathcal T^\mp_{\bf n}: H^{1}_{\bf L}(\Omega^\mp)^3
\to  \Hnh(\Gamma)^3$,
with $H^{1}_{\bf L}(\Omega)^3 = \{ u\in H^{1}(\Omega)^3 \;|\; \mathbf L u \in L^2_{\mathop {\rm loc}}(\Omega)^3 \}$, based on Green's first identity.
We then define the corresponding jump operator by
\bel{jump-operator2}
\jump{\mathcal T \varphi} =\mathcal T^-_{\bf n} \varphi -\mathcal T^+_{\bf n}\varphi.
\ee
Further, define $\mathcal T_{\bf n}: H^1_{\bf L}(\mathbb R^3 )^3\to  \Hnh(\Gamma)^3$ the global normal derivative operator given by $\mathcal T_{\bf n} \varphi =\mathcal T^-_{\bf n} \varphi =\mathcal T^+_{\bf n}\varphi$.

\subsection{Fundamental solutions}
Consider the matrix-valued  fundamental solution $G= (G_{ij})_{ij}$  to the linear isotropic elasticity equation such that $G_i$, the $i$-th column of the matrix $G$  satisfies the following identity: 
\bel{fundamental} 
{\bf L}G_i (x)= \delta(x) \, {\bf e}_i,
\ee
with  ${\bf L}$ defined by \eqref{elasticity},  $\delta $  being the Dirac distribution at the origin and  ${\bf e}_i$ the   canonical basis in $\mathbb R^3$. The Green's function  $G$ is given by~\cite{McLean,steinbach2007numerical}: 
\bel{Green}
G_{ij}(x):= {1 \over 8 \pi \mu  |x|} \left({\lambda+3\mu \over \lambda+2\mu}   \delta_{ij} +{\lambda+\mu \over \lambda+2\mu}{ x_ix_j  \over |x| ^2  }  \right), 
\ee
where we recall that $\mu, \lambda$ are the Lam\'e constants and $\delta_{ij}$ is the Kronecker symbol.
\subsection{Rigid displacement}
For a given domain $\Omega\subset \mathbb R^3$, we consider the following problem: find $u \in H^1_{\mathop{\rm loc}}(\Omega)^3$ such that 
\bel{simple}
{\bf L} u =0,\quad\mbox{in}~H^{-1}(\Omega)^3
\ee
with appropriate boundary conditions. Equation \eqref{simple} holds obviously if  $e(u)=0$. Indeed, 
we call the displacement $u\in H^1_{\mathop{\rm loc}}(\Omega)^3$ a \emph{rigid displacement} if $e(u)=0$. It is well-known that the displacement~$u$ is a rigid displacement if and only if it has the form $u=Ax+b$ where $A\in \mathbb R^{3\times 3}$ is a constant skew matrix  and $b\in\mathbb R^3$ a constant vector (see, for example, \cite{JKO,McLean}).   
\section{Layer potentials}
\label{sec:layer-potentials}
In this section, we introduce the layer potentials and associated boundary operators which only have been sparsely reported in the literature for the operator ${\bf L}$. We therefore provide a complete overview. 
\subsection{Single layer potentials}

Using the fundamental solution \eqref{Green}, we can now define the single layer potential $\SL: \Hnh(\Gamma)^3 \to H^1 (\mathbb R^3 \textbackslash \Gamma)^3 $  associated to  the isotropic elasticity operator ${\bf  L}$: 
\bel{single-layer}
( \SL \phi) (x):= \int_\Gamma G(x-y) \, \phi(y ) \, dy, \quad x \in \mathbb R ^3\textbackslash \Gamma.   
\ee
Further, see e.g.,~\cite{BUI}, such function $\SL \phi $ defined on $\mathbb R^3 \textbackslash \Gamma$ is continuous across the interface $\Gamma$, i.e. $\llbracket \SL \phi \rrbracket = 0 $  and a single layer boundary operator  $\SLO: \Hnh(\Gamma) ^3\to \Hh(\Gamma) ^3$ can  be defined by restricting the single layer potential to the boundary $\Gamma$:
\bel{single-layer-boundary}
(\SLO \phi) (x) := \int_\Gamma G(x-y) \, \phi (y) \, dy, \quad x \in \Gamma,  
\ee
so that $\gamma \SL \phi =  \SLO \phi$. 
The following result is obvious: 
\begin{lemma}
\label{lemma-single}
For $ \phi\in \Hnh(\Gamma)^3$ and $\SL \phi $ defined by \eqref{single-layer}, let ${\bf  L}$ be the isotropic elasticity operator \eqref{elasticity} and 
 we have 
\[
{\bf L}\SL \phi= 0 \quad \mbox{in } \mathbb R^3\textbackslash\Gamma. 
\]
\end{lemma}

\subsection{Double layer potential}
We  introduce the double layer potential $\DL:\Hh(\Gamma)^3\to H^1(\mathbb R^3\backslash \Gamma)^3$,  by
\bel{double-layer}
\DL\varphi (x) =  \int _\Gamma \mathcal T_{{\bf n},y}(G) (x-y) \,  \varphi (y) \, dy, \quad x\in \mathbb R^3\backslash \Gamma,
\ee
where  the subscript $y$ means that the normal derivative operator $\mathcal T_{\bf n}$, defined in Section~\ref{ssec:not}, is taken with respect to the $y$-variable. We define the double layer boundary operator $\DLO :\Hh(\Gamma)^3\to \Hh(\Gamma)^3$ by 
\bel{double-boundary}
(\DLO \varphi)(x) =  \int _ \Gamma \mathcal T_{{\bf n},y} (G) (x-y) \,  \varphi(y) \, dy, \quad x\in \Gamma, 
\ee
in the sense of principal value.  Further, the adjoint double layer boundary operator  $\DLO ^*:\Hnh(\Gamma)^3\to \Hnh(\Gamma)^3$ is given as
\bel{double-boundary*}
(\DLO^*\phi)(x) = \int _ \Gamma \left(\mathcal T_{{\bf n}, x} (G) \right)^\top (x-y) \, \phi(y) \, dy, \quad x\in \Gamma,
\ee 
Similar to Lemma~\ref{lemma-single}, the following result is obvious: 
\begin{lemma}
\label{lemma-double}
For $ \varphi\in \Hh(\Gamma)^3$ and $\DL \varphi $ defined by \eqref{double-layer}, we have 
\[
{\bf  L}\DL \varphi= 0 \quad \mbox{in } \mathbb R^3\textbackslash\Gamma, 
\]
where ${\bf  L}$ is the isotropic elasticity operator \eqref{elasticity}.
\end{lemma}

\subsection{Newton potential}
Finally, for sake of completeness, we also give the  Newton potential associated to the isotropic elasticity operator \eqref{elasticity}. Define $ \mathcal N: H_{\rm comp}^{s}(\mathbb R^3)^3\to H_{\rm loc}^{s+2}(\mathbb R^3)^3$ for $s\in\mathbb R$:
\bel{Newton}
\mathcal N \psi(x) = \int_{\mathbb R^3} G(x-y) \psi (y) \, dy, \quad x\in \mathbb R^3,  
\ee
where $G$ is the Green's function defined by \eqref{Green}. Following the definition for the elasticity operator ${\bf L}$ and all $\psi \in \mathcal D'(\mathbb R^3)^3$, we have 
\bel{equivalent}
\psi={\bf L} \mathcal N \psi =  \mathcal N{\bf L}\psi, \quad \mbox{in } \mathcal D'(\mathbb R^3)^3.  
\ee
Let $\gamma^*: \Hnh(\Gamma)^3\to H^{-1}_{\mathop{\rm comp}}(\mathbb R^3)^3$, $\mathcal {T_{\bf n}}^*: \Hh(\Gamma) ^3\to  H^{-1}_{\mathop{\rm comp}}(\mathbb R^3\textbackslash \Gamma)^3$ be the  adjoint of the trace operator $\gamma$ and the adjoint of the normal derivative operator $\mathcal {T_{\bf n}}$ respectively, defined in Section~\ref{ssec:not}. We then give an equivalent definition of the single and double layer potential: 
\bel{other-def}
\SL = \mathcal N\gamma^*, \quad \DL = \mathcal N {\mathcal T_{\bf n}^*}. 
 \ee  

\subsection{Properties of layer potentials}
\label{sec:properties}
We are now listing a selection of known results of layer potentials that will be used in the following.
Let us first recall the following theorem given in \cite{BUI} (see also \cite[Section 6.7]{steinbach2007numerical}): 
\begin{theorem}
\label{double-in/out}
Let $\phi\in \Hnh(\Gamma)^3$   and  the single layer potential $\SL$, the adjoint double layer boundary operator $\DLO^*$ be defined by \eqref{single-layer}, \eqref{double-boundary*} respectively. Then the interior and exterior normal traces of the stress tensor satisfy 
 \bel{boundary-neuman}
 \mathcal T_{\bf  n}^- \SL \phi=     {1\over 2 } \phi+  \mathcal \DLO^*\phi, \qquad   
 \mathcal T_{\bf  n}^+ \SL \phi =  -{1\over 2 } \phi+  \mathcal \DLO^*\phi, \qquad \mbox{on } \Hnh(\Gamma)^3. 
 \ee
\end{theorem}
We now show several jump conditions relating to the boundary layer potentials above which can be found, for example, in \cite[Theorem 6.10]{McLean}.
\begin{theorem}
\label{Th:jump-relations}
Let $\Omega^-\in \mathbb R^3$ be a bounded Lipschitz domain with  boundary $\Gamma$.  Consider the single and double layer potentials defined by \eqref{single-layer} and \eqref{double-layer} respectively. Then it holds 
\bel{jump-relation}
\begin{array}{l}
\llbracket \SL \phi \rrbracket = 0, \quad \llbracket \DL \varphi \rrbracket = -\varphi, \quad \mbox{on } \Hh(\Gamma)^3, \\
\llbracket \mathcal T \SL \phi \rrbracket  =  \phi, \quad \llbracket \mathcal T \DL \varphi \rrbracket  = 0, \quad \mbox{on } \Hnh(\Gamma)^3
\end{array}
\ee
for all $\phi\in \Hnh(\Gamma)^3, \varphi\in \Hh(\Gamma)^3$. 
\end{theorem}

We now consider the invertibility of the single  layer boundary operator \eqref{single-layer-boundary} (see \cite[Theorem 10.7]{McLean} or \cite[Theorem 6.36]{steinbach2007numerical}.
\begin{lemma}
\label{positive}
Let $\Omega^- \subset \mathbb R^3 $ be a bounded domain with  Lipschitz boundary $\Gamma$. 
If $\mu>0$ and $\lambda\ge 0$, the single layer boundary operator $\SLO$ defined by \eqref{single-layer-boundary} is coercive, i.e.
\[
	\langle  \SLO\phi ,\phi \rangle_{\Hh(\Gamma)\times \Hnh(\Gamma)}
	>c \, \|\phi\|^2_{\Hnh(\Gamma)},\qquad\forall \phi \in \Hnh(\Gamma)^3.
\]
\end{lemma}

\begin{corollary}[Invertibility of the single layer boundary operator]
\label{Corol:invertible}
Let $\Omega^- \subset \mathbb R^3 $ be a bounded domain with  Lipschitz boundary $\Gamma$ and $\mu>0$ and $\lambda\ge 0$. Then, the  single layer boundary operator $\SLO:\Hnh(\Gamma )^3
 \to \Hh(\Gamma)^3 $  is invertible. 
\end{corollary}

\section{Real vector spherical harmonics}
\label{sec:VSH}

\subsection{Surface gradient}
In the following, we introduce the real vector spherical harmonics. We begin with some conventions of the gradient. On a given domain $\Omega$, consider  a scalar valued function $f:\Omega\to \mathbb R^3$ and  a column-vector valued  function $F:\Omega\to  \mathbb R^3$, we define 	their gradients by
\bel{gradient-vector}
\nabla f(x)\in \mathbb R^3,~\mbox{with}~(\nabla f)_i = { \del f\over \del x_i}, \qquad
\nabla F(x) \in \mathbb R^{3 \times 3 }, ~\mbox{with}~(\nabla F)_{ij} ={ \del F_i\over \del x_j}. 
\ee
 Note in particular that $\nabla f$ is a column-vector while $\nabla F$ are row-wise gradients for each component $F_i$.

Restricting the considerations to the unit ball $\Omega=B$ and it surface $\partial\Omega=\mathbb S^2$, we denote by 
\bel{surface1}
	\nabla_{\! \rm s}=\hat \theta {\del\over \del\theta}+\hat \phi{1\over \sin\theta}{\del\over \del \phi}
\ee
 the surface gradient operator and $\hat r, \hat \theta, \hat \phi$ are radial, polar and azimuthal unit vectors which are supposed to be row vectors. Let $f$ denote a scalar function and $F$ a vector-valued function, i.e. $f\in \mathbb R$ and $F\in \mathbb R^3$ and with the convention of the gradient field, the surface gradient \eqref{surface1} can alternatively be written as
\bel{surface}
\begin{array}{l}
\nabla_{\! \rm s} f = \nabla f- {\bf n} ({\bf n} ^\top \nabla f),\\
\nabla_{\! \rm s} F = \nabla F- ( \nabla  F {\bf n}) {\bf n} ^\top,\\
 \end{array}
\ee
where $\nabla$ is the gradient in $\mathbb R^3$ based on the convention~\eqref{gradient-vector}. It is immediate to verify that 
\[
\nabla_{\! \rm s} f ^\top  {\bf n} =0, \quad \nabla_{\! \rm s}F {\bf n}=0. 
\]

\subsection{Definition of vector spherical harmonics}
The construction of the vector spherical harmonics is based on the 
scalar real spherical harmonics defined on the unit sphere $\mathbb S^2$ denoted by $(Y_{\ell m})_{l\geq 0}^{|m|\leq \ell}$ which are normalized such that 
\[
\langle Y_{\ell m}, Y_{\ell'm'}\rangle_{\mathbb S^2} = \int_{\mathbb S^2} Y_{\ell m} Y_{\ell'm'} = \delta_{\ell\ell'}\delta_{m'm'}.
\]
The vector spherical harmonics $V_{\ell m},W_{\ell m},X_{\ell m}: \mathbb S^2 \to \mathbb R^3 $ of degree $\ell\geq 0$ and order $m$, $|m|\leq\ell$  are given by 
\bel{vector-harmonics}
\aligned 
& V_{\ell m} := \nabla_{\! \rm s} Y_{\ell m}(\theta,\phi) - (\ell+1)Y_{\ell m}(\theta,\phi) \hat r,\\
& W_{\ell m} :=\nabla_{\! \rm s}  Y_{\ell m}(\theta,\phi)+\ell Y_{\ell m}(\theta,\phi) \hat r, \\
& X_{\ell m} :=\hat r \times \nabla_{\! \rm s} Y_{\ell m} (\theta,\phi).
\endaligned
\ee
The  symbol $\times$ represents the cross product in $\mathbb R^3$.  
We refer to  Appendix A for some explicit expressions of the vector spherical harmonics for the first few degrees. 
The vector spherical harmonics satisfy the following orthogonal properties: 
\bel{orthogonal}
\aligned 
&\int_{\mathbb S^2 } V_{\ell m} \cdot W_{\ell'm'} =0, ~\int_{\mathbb S^2} W_{\ell m}  \cdot  X_{\ell'm'}=0, ~ \int_{\mathbb S^2} X_{\ell m}   \cdot  V_{\ell'm'}=0, 
\\
&\int_{\mathbb S^2} V_{\ell m}  \cdot  V_{\ell'm'}=\delta_{\ell\ell'}\delta_{m'm'} (\ell+1) (2\ell+1), ~ \int_{\mathbb S^2} W_{\ell m}  \cdot  W_{\ell'm'}= \delta_{\ell\ell'}\delta_{m'm'} \ell(2\ell+1), \\& \int_{\mathbb S^2} X_{\ell m}  \cdot  X_{\ell'm'}= \delta_{\ell\ell'}\delta_{m'm'} \ell(\ell+1). 
\endaligned 
\ee
The scalar spherical harmonics (and thus the vector spherical harmonics) can be extended to any sphere $\Gamma_r(x_0) = \partial B_r (x_0)$ by translation and scaling. 
We will introduce the following scaled scalar product on $\Gamma_r(x_0)$ given by 
\bel{inner-product}
\langle u, v\rangle_{\Gamma_r(x_0)}= {1\over r^2}\int_{\Gamma_r(x_0)}u(s)\cdot v(s)ds= \int_{\mathbb S^2} u(x_0+ rs')\cdot v(x_0+rs')ds'. 
\ee
In practice, the  exact value of the scalar product \eqref{inner-product} cannot be computed explicitly in general. With a set $\{s_t, w_t\}_{t=1}^{T_g}$ of integration points and weights on the unit sphere,  the scalar product is approximated  by the quadrature rule
 \bel{quadrature}
 \langle u, v\rangle_{\Gamma_r(x_0), t} = \sum\limits_{t=1} ^{T_g} w_t \, u(x_0+ r s_ t )\cdot v(x_0+ r s_ t ). 
 \ee
 In the numerical tests below in Section~\ref{sec:8},  we will use the Lebedev quadrature points ~\cite{integral},
 which have the property that scalar spherical harmonics up to a certain degree $N_g$ are integrated exactly. This relationship is displayed in Table~\ref{tab:Lebedev}. 
 It can be noticed that the number of points increases quadratically with $N_g$.

Further, the family of vector spherical harmonics gives a complete basis of $L^2(\Gamma_r(x_0))^3$ and any real function $f\in L^2(\Gamma_r(x_0))^3$ can be represented as
\bel{expand}
	f(x)
	=
	\sum\limits_{\ell=0}^\infty\sum\limits_{m=-\ell}^\ell [v]_{\ell m} V_{\ell m} \Big({x-x_0\over  r }\Big) +  [w]_{\ell m}W_{\ell m} \Big({x-x_0\over  r}\Big)+ [x]_{\ell m} X_{\ell m} \Big({x-x_0\over  r }\Big), 
\ee
where $[v]_{\ell m}, [w]_{\ell m}, [x]_{\ell m} \in \mathbb R$. 

\begin{table}
\label{tab:Lebedev}
\centering
{\footnotesize
\begin{tabular}{||c|cccccccccccccc||} \hline\hline
$N_g$ & 3 & 5 & 7 & 9 & 11 & 13& 15 & 17 & 19 & 21 & 23 & 25 & 27 & 29 \\ \hline
$T_g$ & 6 & 14 & 26 & 38 & 50 & 74 & 86 & 110 & 146 &  170&  194 &  230 &  266 &  302  \\ \hline\hline 
$N_g$ & 31 & 35 & 41 & 47 & 53 & 59 & 65 & 71 & 77 & 83 & 89 & 95 & 101 & 107 \\ \hline
$T_g$ &  350 &  434 &  590 &  770 &  974 & 1202 & 1454 & 1730 &  2030  &  2354 & 2702 & 3074 & 3470  &  3890 \\ \hline\hline 
\end{tabular}
}
\caption{Degree $N_g$ and number of points $T_g$ of Lebedev quadrature rules such that spherical haromonics up to degree $N_g$ are integrated exactly with $T_g$ points.}
\end{table}

\subsection{Properties of the derivatives}
We give some derivative properties of the surface gradient \eqref{surface} that shall be useful in the upcoming analysis. In the following, let $u$ be the a scalar-valued function, $F,G$ be vector-valued functions and ${\bf  A } $ be a matrix-valued function.  We have the following product rule: \bel{product-rules}
\begin{array}{l}
\nabla_{\! \rm s} (uF)=      F   \nabla_{\! \rm s}  u ^\top   + u \nabla_{\! \rm s} F, \\
\nabla_{\! \rm s}(F^\top G)=  \nabla_{\! \rm s} F^\top G  +  \nabla_{\! \rm s} G^\top F. 
\end{array}
\ee
We also have the property for the cross product: 
\bel{gradient-cross}
\nabla_{\! \rm s} (F\times G)= \nabla_{\! \rm s} F\times G -  \nabla_{\! \rm s} G\times F. 
\ee
Later proof also requires the triple product
\bel{cross-inner}
\big( {\bf  A } \times F \big)^\top G =  {\bf  A }^\top \big( F\times G \big) =  \big( G\times {\bf  A} \big)^\top F. 
\ee
Now let $h= h(r)$ be a scalar function which does not depend on the polar angles and $u= u(\theta,\phi), H=H(\theta,\phi)$ a scalar and a vector valued function respectively depending only on the polar angles. Then there holds 
 \bel{grad}
 \begin{array}{l}
 \nabla (hu)= h_r    u \hat r   + {1\over r} h \nabla_{\! \rm s} u, \\
\nabla (hH)= h_r    H \hat r ^\top  + {1\over r} h \nabla_{\! \rm s} H. 
\end{array}
\ee
For the scalar function $h= h(r)$, denote by $h_r, h_{rr}$ the first and second derivative. Then we have 
\bel{divVSH}
\aligned 
& \mathop{\rm div}(hV_{\ell m})= -(\ell +1) \Big(h_r+{ \ell +2\over r} h \Big) Y_{\ell m},\\
&\mathop{\rm div}(hW_{\ell m})= \ell  \Big(h_r-{ \ell -1\over r} h \Big) Y_{\ell m},\\
&\mathop{\rm div}(hX_{\ell m})=0, 
\endaligned
\ee
and 
\bel{laplaceVSH}
\aligned 
& \Delta (h V_{\ell m})= \Big(h_{rr}+{ 2\over r} h_r - {(\ell +1)(\ell +2)\over r^2} h  \Big) V_{\ell m},\\
&\Delta(h W_{\ell m})=  \Big(h_{rr}+{ 2\over r} h_r- {(\ell-1)\ell \over r^2} h  \Big) W_{\ell m},\\
&\Delta(h X_{\ell m})=  \Big(h_{rr}+{ 2\over r} h_r - {\ell (\ell +1)\over r^2} h  \Big) X_{\ell m}.
\endaligned
\ee
Equations~\eqref{divVSH}, \eqref{laplaceVSH} are given in~\cite{Hill}. 
Finally, the following identities hold, resulting  directly from the definition of the vector spherical harmonics: 
\bel{VtoS}
\nabla_{\! \rm s} Y_{\ell m} = {1\over 2\ell+1} (\ell  V_{\ell m}+ (\ell +1)W_{\ell m}), \quad Y_{\ell m}\hat r=  {1\over 2\ell+1} ( W_{\ell m}-V_{\ell m}). 
\ee

\section{Spectral properties of the layer potentials}

\label{sec:spectral}
We will give the main results in Section~\ref{sec:spectral1}, prepare some preliminary results in Section~\ref{sec:spectral2} and finally provide the proofs in Section~\ref{sec:spectral3}. 
\subsection{Main results}
\label{sec:spectral1}
Consider the single layer potential $\SL$ and the single layer boundary operator $\SLO$ defined by \eqref{single-layer} and \eqref{single-layer-boundary}, we have the following result. 
\begin{theorem}
\label{H&N}
  Let $\underline {Y_{\ell m}}$ be the matrix such that 
  \bel{marix-VSH}
\underline {Y_{\ell m}}= (V_{\ell m}|W_{\ell m}|X_{\ell m}):= (Y_{\ell m}^1|Y_{\ell m}^2|Y_{\ell m}^3). 
  \ee
Then we have: 
\begin{enumerate}
\item On the unit sphere $\mathbb S^2$, 
 \[
\SLO \underline {Y_{\ell m}}(x) = \underline {Y_{\ell m}}  A_{\SLO,\ell}, 
\]
where $\SLO$ is the single layer boundary operator  defined by \eqref{single-layer-boundary}  and $ A_{\SLO,\ell}$ is a constant matrix given by \bel{aon}
\aligned
 A_{\SLO,\ell}& = \begin{bmatrix}
      { (3\ell+1)\mu +\ell\lambda\over  ( 2\ell+3)(2\ell+1) \mu(2\mu+\lambda)}  &0  & 0 \\
    0 &{  (3\ell+2) \mu+ (\ell+1)\lambda \over (2\ell-1)(2\ell+1) \mu(2\mu+\lambda)} & 0 \\
    0&0& {1\over \mu (2\ell+1)}
\end{bmatrix} =  \mathop{\rm diag}(\tau_{\SLO,\ell}^1,\tau_{\SLO,\ell}^2, \tau_{\SLO,\ell}^3 ). 
\endaligned 
\ee
\item When  $|x| <1$, we have 
\[
(\SL  \underline {Y_{\ell m}})(x)=  \underline {Y_{\ell m}}   \left({x\over |x|}\right)A_{\SL,\ell}^{in}(x), 
\]
where  $\SL$ is  the single layer potential given by \eqref{single-layer} and the matrix  $  A_{\SL,\ell}^{in}(x)$ has the form
\bel{ain}
   A_{\SL,\ell}^{in}(x) = \begin{bmatrix}
   { (3\ell+1)\mu +\ell \lambda\over  ( 2\ell+3)(2\ell +1) \mu(2\mu+\lambda)} |x|^{\ell+1}  &0 & 0 \\
    {(\ell+1)(\mu+ \lambda)\over 2(2\ell+1)\mu (2\mu+ \lambda )} (|x| ^{\ell+1}- |x|^{\ell- 1})  & {  (3\ell+2)\mu+ (\ell+1)\lambda\over (2\ell-1)(2\ell+1) \mu(2\mu+\lambda)} |x|^{\ell-1} & 0 \\
    0&0& { 1\over(2\ell+1)\mu}|x|^\ell
\end{bmatrix}.
\ee
\item When $|x|>1$, we have 

\[
(\SL\underline {Y_{\ell m}} )(x)=  \underline {Y_{\ell m}} \left({x\over |x|}\right) A_{\SL,\ell}^{out}(x), 
\]
where the matrix $A_{\SL,\ell}^{out}(x)$ is given by  
\bel{aout}
 A_{\SL,\ell}^{out}(x) = \begin{bmatrix}
    {(3\ell+1)\mu +\ell \lambda \over ( 2\ell+3)(2\ell+1) \mu(2\mu+\lambda)}  |x|^{-\ell-2}  & {  \ell  (\mu+ \lambda )\over 2 (2\ell+1) \mu(2\mu+\lambda)}  (|x|^{-\ell-2}-|x|^{-\ell}) & 0 \\
  0 &  {  (3\ell+2)\mu+ (\ell+1)\lambda \over (2\ell-1)(2\ell+1) \mu(2\mu+\lambda)}  |x|^{-\ell} & 0 \\
    0&0& { 1\over(2\ell+1)\mu}|x|^{-\ell-1}
\end{bmatrix}. 
\ee
\end{enumerate}
\end{theorem}
The following result is a corollary of Theorem~\ref{H&N}. 
\begin{corollary}
\label{corollaryH&N}
 Let $\underline {Y_{\ell m}}$ be the a matrix defined by  \eqref{marix-VSH}. 
Then, on the unit sphere $\mathbb S^2$, there holds
 \[
\DLO ^*  \underline {Y_{\ell m}} = \DLO  \underline {Y_{\ell m}} = \underline {Y_{\ell m}}  A_{ \DLO ^*,\ell},
\]
where $ \DLO $ is the  double layer boundary operator \eqref{double-boundary*}  with its adjoint $\DLO ^* $ and $A_{\DLO^*,\ell}$ is a constant and diagonal matrix: 
\bel{adon}A_{\DLO^*,\ell} = \begin{bmatrix}
      -  {2(2\ell^2+6\ell+1)\mu -3 \lambda\over 2 (2\ell+1)(2\ell+3) (2\mu+\lambda)}  &0  & 0 \\
    0 & {2(2\ell^2-2\ell-3 )\mu-  3 \lambda\over  2 (2\ell+1)(2\ell-1) (2\mu+\lambda)}   & 0 \\
    0&0& {1\over 2\mu (2\ell+1)}
\end{bmatrix} = \mathop{\rm diag}(\tau_{\DLO ^*,\ell}^1,\tau_{\DLO ^*,\ell}^2, \tau_{\DLO ^*,\ell}^3 ). \ee
\end{corollary}

\begin{remark}
\label{-1/2}
Recall that the Lam\'e constats $\mu, \lambda$ satisfy $\mu>0, 2\mu+3\lambda>0$, we can verify that the eigenvalue $\tau_{\DLO ^*,\ell}^k$ of the adjoint double layer boundary operator  is $-{1\over 2}$ if and only if $\ell=1, k=2$. And the eigenvectors associated with the eigenvalue $-{1\over 2}$ are $W_{1m}$, $m=\pm 1,0$. 
\end{remark}
The following theorem gives explicit expressions of the double layer potential. 
\begin{theorem}
\label{H&N'}
Let $\underline {Y_{\ell m}}$ be  given by \eqref{marix-VSH} and $\DL$ the double layer potential on the unit sphere introduced in \eqref{double-layer}, we have:
\begin{enumerate}
\item For $|x|<1$,
\[
(\DL \underline{Y_{\ell m}})(x) =  \underline {Y_{\ell m}} \Big({x\over |x|}\Big) A_{\DL,\ell}^{in}(x)
\]
where 
\bel{ain'}
   A_{\DL,\ell}^{in}(x) = \begin{bmatrix}
a^{in,\DL,\ell}_{11}  |x|^{\ell+1}  &a^{in,\DL,\ell}_{12}   |x|^{\ell+1}& 0 \\
 a^{in,\DL,\ell}_{21,1} |x| ^{\ell+1}+ a^{in,\DL,\ell}_{21,2} |x|^{\ell-1} &a^{in,\DL,\ell}_{22,1} |x|^{\ell+1}+a^{in,\DL,\ell}_{22,2}  |x|^{\ell-1}& 0 \\
    0&0& -{\ell+1\over(2\ell+1)\mu}|x|^\ell
\end{bmatrix}.
\ee
\item For $|x|>1$,
\[
(\DL \underline{Y_{\ell m}})(x) =  \underline {Y_{\ell m}} \Big({x\over |x|}\Big)A_{\DL,\ell}^{out}(x)
\]
where 
\bel{aout'}
   A_{\DL,\ell}^{out}(x) = \begin{bmatrix}
a^{out,\DL,\ell}_{11,1}|x|^{-\ell-2} +a^{out,\DL,\ell}_{11,2}|x|^{-\ell}&a^{out,\DL,\ell}_{12,1} |x|^{-\ell-2} + a^{out,\DL,\ell}_{12,2}|x|^{-\ell} & 0 \\
    a^{out,\DL,\ell}_{21} |x|^{-\ell}  &a^{out,\DL,\ell}_{22}  |x|^{-\ell}& 0 \\
    0&0& { \ell\over(2\ell+1)\mu}|x|^{-\ell-1}
\end{bmatrix}.
\ee\end{enumerate}
The constants $a^{in,\DL,\ell}_{ij}, a^{out,\DL,\ell}_{ij}  $ are listed in Appendix B. 
\end{theorem}
The results of Theorems~\ref{H&N}, \ref{H&N'} and Corollary~\ref{corollaryH&N}
 can be extended to any sphere $\Gamma_r(x_0) = \partial B_r(x_0)$, which is, a sphere centered at $x_0$ with radius $r_0$,  by setting $\underline {Y_{\ell m}}\Big({x-x_0\over|x-x_0| }\Big)$ and using the following scaling in $r$: 
\bel{H&N-scaling}
\begin{array}{l}
(\SL \underline {Y_{\ell m}})(x)= r\underline {Y_{\ell m}}\Big({x-x_0\over|x-x_0| }\Big)A_{\SL,\ell}^{in/out}\Big({x-x_0\over r }\Big), \quad( \DL \underline {Y_{\ell m}})(x) = \underline {Y_{\ell m}}\Big({x-x_0\over|x-x_0| }\Big)A_{\DL,\ell}^{in/out}\Big({x-x_0\over r }\Big), \\
(\SLO \underline {Y_{\ell m}})(x)= r\underline {Y_{\ell m}}\Big({x-x_0\over|x-x_0| }\Big)A_{\SLO,\ell}, \quad (\DLO^*\underline {Y_{\ell m}})(x) =(\DLO\underline {Y_{\ell m}})(x) =  \underline {Y_{\ell m}}\Big({x-x_0\over|x-x_0| }\Big)A_{\DLO^*,\ell}. 
\end{array}
\ee

\subsection{Preliminary lemmas}
\label{sec:spectral2}
To prove the results of Section~\ref{sec:spectral1} in the upcoming Section~\ref{sec:spectral3}, we derive first several preliminary lemma.
\begin{lemma}
For a scalar function $h=h(r)$, we have the following identities
\bel{T_N}
\aligned  &2e\Big( h(r)V_{\ell m} \Big){\bf n}| _{\mathbb S^2}
=\Big({( 3\ell+2)h_r (1)- \ell ( \ell+2) h(1)\over 2\ell+1}\Big)V_{\ell m}+ \Big({-(\ell+1)h_r(1)-(\ell+1)(\ell+2)h(1)\over 2\ell+1} \Big)W_{\ell m}, \\
& 2e \Big(h(r)W_{\ell m}\Big) {\bf n}| _{\mathbb S^2}
=  \Big({-\ell h_r (1)+\ell(\ell -1) h(1)\over 2\ell+1}\Big)V_{\ell m}+ \Big({(3\ell+1)h_r(1)+(\ell-1)(\ell+1) h(1)\over 2\ell+1} \Big)W_{\ell m}, 
\\  & 2e\Big(h(r)X_{\ell m})\Big) {\bf n}| _{\mathbb S^2}= \Big(h_r(1)-h(1) \Big)X_{\ell m}, 
\endaligned 
\ee
 where ${\bf n}$ is the outward pointing unit normal vector of the unit sphere $\mathbb S^2$.  
\end{lemma}

\begin{proof}
We consider $h(r) V_{\ell m}$ first.  Following \eqref{product-rules} and \eqref{grad}, we have, 
\bel{DV}
\aligned 
 \Big(\nabla  (h(r)V_{\ell m})+ \nabla  (h(r)V_{\ell m})^\top \Big){\bf n}|_{\mathbb S^2}= & h_r(1)\big( V_{\ell m}  \hat r ^\top +  \hat r V_{\ell m}^\top\Big)  {\bf n}+ h(r)(\nabla_{\! \rm s} V_{\ell m}+ \nabla_{\! \rm s} V_{\ell m}^\top) {\bf n}| _{\mathbb S^2}\\ 
 = & h_r(1)\big( V_{\ell m}+ \hat r V_{\ell m}^\top \big)\hat r+ h(1)(\nabla_{\! \rm s} V_{\ell m}+ \nabla_{\! \rm s} V_{\ell m}^\top) \hat r,
 \endaligned 
\ee
where we also use the fact that ${\bf n}$ is  equal to the radial  basis $\hat r $ on the unit sphere $\mathbb S^2$. 
Consider now the first term $ h_r(1)( V_{\ell m}+ \hat r V_{\ell m}^\top )  \hat r $. To compute $ \hat r V_{\ell m}^\top   \hat r$, we need first the  relation $\nabla_{\! \rm s}  Y_{\ell m}^\top \hat r=0$ following from \eqref{surface}.  Then according to the definition of the vector spherical harmonics $V_{\ell m}$ given by \eqref{vector-harmonics}, we have 
\bel{turn}
(\hat r  V_{\ell m}^\top) \hat r= ( \hat r \nabla_{\! \rm s}  Y_{\ell m}^\top ) \hat r - (\ell+1)  Y_{\ell m} \hat r \hat r ^\top \hat r= -(\ell+1) Y_{\ell m} \hat r=  {(\ell+1) \over 2\ell+1} ( -W_{\ell m}+V_{\ell m}). 
\ee
Therefore, the first term in \eqref{DV} yields 
\[
 V_{\ell m}+ \hat r V_{\ell m}^\top  \hat r={ 3\ell+2\over 2\ell+1 }V_{\ell m} -{\ell+1\over 2\ell+1}W_{\ell m}. 
\]
Now consider the second term $(\nabla_{\! \rm s} V_{\ell m}+ \nabla_{\! \rm s} V_{\ell m}^\top) \hat r$.  According to \eqref{surface}, we have $\nabla_{\! \rm s} V_{\ell m}  \hat r=0$.  The definition of $V_{\ell m}$ \eqref{vector-harmonics} gives 
\[
\aligned 
\nabla_{\! \rm s} V_{\ell m} ^\top  \hat r= \nabla_{\! \rm s}( \nabla_{\! \rm s}Y_{\ell m})^\top\hat r  -(\ell+1) \nabla_{\! \rm s}(Y_{\ell m} \hat r )^\top\hat r.
\endaligned 
\]
We compute the two terms separately. Using the product rule  \eqref{product-rules}, we have 
\[
\nabla_{\! \rm s}( \nabla_{\! \rm s}Y_{\ell m})^\top \hat r= \nabla_{\! \rm s}( \nabla_{\! \rm s}Y_{\ell m} ^\top \hat r) -  \nabla_{\! \rm s} \hat r^\top  \nabla_{\! \rm s} Y_{\ell m}= -(\mathop{\rm Id} - \hat r \hat r ^\top)^\top \nabla_{\! \rm s} Y_{\ell m}= -\nabla_{\! \rm s} Y_{\ell m}. 
\]
For  $  \nabla_{\! \rm s}(Y_{\ell m} \hat r )^\top \hat r $, we use  \eqref{product-rules} and have 
\[
 \nabla_{\! \rm s}(Y_{\ell m} \hat r )^\top \hat r =  \nabla_{\! \rm s} Y_{\ell m}  \hat r^\top   \hat r+ (Y_{\ell m}  \nabla_{\! \rm s} \hat r ) ^\top \hat r =  \nabla_{\! \rm s}Y_{\ell m} + Y_{\ell m} \nabla_{\! \rm s} \hat r^\top  \hat r=  \nabla_{\! \rm s}Y_{\ell m}, 
\]
where we use the fact that $\nabla_{\! \rm s} \hat r$ is a symmetric matrix  and $\nabla_{\! \rm s} \hat r \hat r=0$. 
 Therefore, there holds
\bel{nabla-V}
\nabla_{\! \rm s} V_{\ell m} ^\top \hat r= -(\ell+2)  \nabla_{\! \rm s}  Y_{\ell m} =-{\ell+2\over 2\ell+1} \big(\ell V_{\ell m}+ (\ell+1)W_{\ell m}\big). 
\ee
The  computation   of $e\big(h(r)V_{\ell m}\big){\bf n}$ is thus completed in view of \eqref{turn} and \eqref{nabla-V}. 
A similar computation gives $h(r)W_{\ell m}$. Consider now $h(r)X_{\ell m}$.  Similar to \eqref{DV}, we have to compute the sum: 
\[
 e\big(h(r)X_{\ell m}\big){\bf n}|_{\mathbb S^2}= h_r(1)\big( X_{\ell m}+ \hat r X_{\ell m}^\top \big)\hat r+ h(1)(\nabla_{\! \rm s} X_{\ell m}+ \nabla_{\! \rm s} X_{\ell m}^\top) \hat r.
\] Notice that  $X_{\ell m}= \hat r \times \nabla_{\! \rm s} Y_{\ell m}$ is orthogonal to $\hat r$ and it follows immediately that $ X_{\ell m}^\top\hat r =0$. Further, by \eqref{surface}, there holds $\nabla_{\! \rm s} X_{\ell m}\hat r=0$. Then, it remains to consider $X_{\ell m}\hat r$ and $\nabla_{\! \rm s} X_{\ell m}^\top \hat r$. For the term   $\nabla_{\! \rm s} X_{\ell m} ^\top \hat r$, we use the relation \eqref{gradient-cross}  and have 
\[
\nabla_{\! \rm s} X_{\ell m} ^\top \hat r= \big(\nabla_{\! \rm s} (\hat r \times \nabla_{\! \rm s} Y_{\ell m} )\big)^\top\hat r=(\nabla_{\! \rm s}  \hat r  \times  \nabla_{\! \rm s} Y_{\ell m})^\top  \hat r+  \big(\hat r\times \nabla_{\! \rm s} (\nabla_{\! \rm s} Y_{\ell m} ) \big)^\top \hat r. 
\]
Both terms  can be computed by \eqref{cross-inner}: 
\[
 \big(   \hat r\times\nabla_{\! \rm s} (\nabla_{\! \rm s} Y_{\ell m} ) \big)^\top \hat r= \nabla_{\! \rm s} (\nabla_{\! \rm s} Y_{\ell m} )  ^\top  (\hat r\times \hat r)   = 0, 
\]
and 
\[
\aligned 
(\nabla_{\! \rm s}  \hat r  \times  \nabla_{\! \rm s} Y_{\ell m})^\top  \hat r = &   (\hat r\times \nabla_{\! \rm s}    \hat r)^\top \nabla_{\! \rm s} Y_{\ell m} = \big( \hat r\times (\mathop{\rm Id} - \hat r \hat r ^\top) \big) ^\top \nabla_{\! \rm s} Y_{\ell m}  \\
= & (  \hat r\times \mathop{\rm Id}  ) ^\top \nabla_{\! \rm s} Y_{\ell m}   =   \nabla_{\! \rm s} Y_{\ell m}\times \hat r.
\endaligned 
\]
This gives
\[
\nabla_{\! \rm s} X_{\ell m}^\top  \hat r= - X_{\ell m}. 
\]
Then we get the result for $h(r) X_{\ell m}$.  
\end{proof}
By \eqref{divVSH}, \eqref{T_N}, we get, for a displacement  $h(r)V_{\ell m}+ g(r)W_{\ell m}+ h(r)X_{\ell m}$, it holds that 
\bel{Tfgh}
\aligned 
 \mathcal T_{\bf n} ^-\big(f(r)&V_{\ell m}+ g(r)W_{\ell m}+ h(r)X_{\ell m}\big)= \mathcal T_{\bf n} ^+\big(f(r)V_{\ell m}+ g(r)W_{\ell m}+ h(r)X_{\ell m}\big)
\\=& \Big({ \mu \over 2\ell+1}\big(( 3\ell+2) f_r (1)- \ell(\ell+2) f(1)-\ell g_r (1)+\ell (\ell -1) g(1)\big)
\\& + {\lambda\over 2\ell+1 }\big((\ell+1 )f_r(1)+ (\ell+1)(\ell+2)f(1)- \ell g_r(1)+ \ell(\ell-1)g(1)\big)\Big) V_{\ell m}
\\& + \Big({ \mu \over 2\ell+1}\big(-(\ell +1) f_r (1)- (\ell+1) (\ell+2) f(1)+(3\ell+1)g_r (1)+(\ell+1)(\ell-1) g(1)\big)
\\& + {\lambda\over 2\ell+1 }\big(-(\ell+1 )f_r(1)-(\ell+1) (\ell+2)f(1)+\ell g_r(1)-\ell (\ell -1)g(1)\Big) W_{\ell m}\\& +\mu  \big(h_r(1)-  h(1)\big)X_{\ell m}. 
\endaligned 
 \ee
The flowing lemma concerns the double layer boundary operator and its adjoint. 
In particular, on a sphere, we have the following lemma. 
\begin{lemma}
\label{stars}
Let $ \DLO $ be the double layer boundary operators defined by \eqref{double-boundary} on a sphere and $\DLO^* $ its adjoint operators.  Then  for $v\in L^2(\mathbb S^2)^3$, we have $\DLO  v =\DLO^*  v $.
\end{lemma}
\begin{proof}
 Indeed, we have 
\[
\aligned 
\del_ {x_k}G_ {ji}(x-y)= & {1 \over 8 \pi\mu |x-y|^3 }\bigg  (-{\lambda+3\mu \over \lambda+2\mu}  \delta _ {ij} (x_ k -y_k)\\ & + {\lambda+\mu \over \lambda+2\mu} \Big ((x_ i-y_i) \delta_ {jk}+ (x_j-y_j) \delta_ {ik}-{ 3 (x_i-y_i)(x_j-y_j)(x_k-y_k)  \over |x-y|^2} \Big) \bigg), 
\endaligned 
\]
and 
\[
  \mathop{\rm Tr}  e_x(G_ j)\mathop{\rm Id}= {\mathop{\rm div}} _x G_ j (x-y )=  - {(x_j -y_ j) \over 4  \pi  |x-y|^3 }  { 1   \over \lambda+2\mu}. \]
Hence, we have 
\begin{align*}
\Big(2\mu e_x( G_j) + &\lambda  \mathop{\rm Tr}  e_x(G_ j)\mathop{\rm Id} \Big)_{ik} {x_k \over |x|} 
\\ 
= &  -{1 \over 4\pi (\lambda +2 \mu)|x-y|^3 }\Big ({\mu \over \lambda+3\mu} ( \delta _ {ij} (x_ k -y_k)+\delta _ {jk} (x_ i -y_i)+\delta _ {ik} (x_ j -y_j))
\\
& +{ 3 (x_i-y_i)(x_j-y_j)(x_k-y_k)  \over |x-y|^2} \Big) \bigg){x_k \over |x|}
\end{align*}
The same result holds for  $ \Big(2\mu e_y(G_i)+ \lambda \mathop{\rm Tr}e_y(  G_ i)\mathop{\rm Id}\Big)_{jk} {y_k \over |y|} $ by replacing $x$ by $y$. 
Further, for $x, y\in \mathbb S^2$, the following relation holds: 
\[
(x-y)\cdot {x \over |x|}= 1- {x\cdot y \over r} = (y-x)\cdot {y \over |y|}. 
\]
Therefore, we have 
\[
\aligned
(\DLO^* v)_i (x) =& \sum\limits_{j} \int_\Gamma {1 \over 4\pi (\lambda +3 \mu)|x-y|^3}\Bigg( \Big(-{\mu \over \lambda+3\mu}  \delta _ {ij}-{ 3 (x_i-y_i)(x_j-y_j)  \over |x-y|^2} \Big)(x-y)\cdot {x \over |x|} v_ j (y)\\
& + {\mu \over \lambda+3\mu} (x_j y_ i -x_i y_j) v_j (y)\Bigg)dy \\
=&  \sum\limits_{j} \int_\Gamma {1 \over 4\pi (\lambda +2  \mu)|x-y|^3}\Bigg (\Big(-{\mu \over \lambda+3\mu}  \delta _ {ij}-{ 3 (y_i-x_i)(y_j-x_j)  \over |x-y|^2} \Big) (y-x)\cdot {y \over |y|} v _ j (y)\\
& + {\mu \over \lambda+3\mu} (x_j y_ i -x_i y_j) v_j (y)\Bigg)dy= (\DLO v)_i (x).
\endaligned
\]
\end{proof}

 \subsection{Proof of  the principal results}
\label{sec:spectral3}
We are now ready to prove  Theorem~\ref{H&N}. 
\begin{proof}[Proof of  Theorem~\ref{H&N}] 
Consider the single layer potential $\SL $  defined by \eqref{single-layer}. Note that for any $\phi\in \Hnh(\mathbb S^2)^3 $, as announced in Lemma~\ref{lemma-single}, $u= \SL \phi$ satisfies the following linear isotropic elasticity system 
\bel{Deltav=0}
{\bf L}u = -\mathop{\rm div}\Big(2 \mu e(u)  +\lambda  \mathop{\rm Tr} e(u)\mathop{\rm Id}\Big) =0, \quad  \text{in $\mathbb R^3 \textbackslash \mathbb S^2$ }. 
\ee  
Now we determine $u=\SL V_{\ell m}$ by means of separation of variables in spherical coordinates. That is, we  propose the Ansatz  displacement field $ \SL V_{\ell m}$ as a function of the spherical coordinates of  form $ u=\SL V_{\ell m} = h(r)V_{\ell m}+g(r)W_{\ell m}+h(r)X_{\ell m}$ where $f,g,h$ are three scalar functions of $r$ to be determined.  Using  the relation 
\[
\mathop{\rm div}( \mathop{\rm Tr} e(u)\mathop{\rm Id})  = \nabla (\mathop{\rm div} u), 
\] and plugging the Ansatz into \eqref{Deltav=0}, together with  \eqref{divVSH}-\eqref{VtoS}, we have the following equation:
\bel{Ode}
\begin{array}{l}
\Bigg[\Big(\mu+ {\ell+1\over 2\ell+1}(\mu+\lambda)\Big)\Big(f_{rr}+ {2\over r} f_r - {(\ell+1)(\ell +2)\over r^2}f\Big)\\\hspace{4cm}- {\ell \over 2\ell+1}(\mu+\lambda) \Big(g_{rr}- {2\ell -1\over r}g_r+ {(\ell-1)(\ell+1)\over r^2} g\Big)\Bigg] V_{\ell m}\\
+\Bigg[\Big(\mu+ {\ell \over 2\ell+1}\Big(\mu+\lambda)\Big)\Big(g_{rr}+ {2\over r} g_r - {(\ell-1)\ell \over r^2}g\Big)- {\ell +1\over 2\ell+1}(\mu+\lambda) \Big(f_{rr}+{2\ell +3\over r}f_r+ {\ell(\ell+2)\over r^2} f\Big)\Bigg] W_{\ell m}\\
+ \mu  \Big[h_{rr}+{ 2\over r} h - {\ell(\ell+1)\over r^2} h  \Big] X_{\ell m}=0.
\end{array}
\ee
Since $V_{\ell m}, W_{\ell m}, X_{\ell m}$ is an orthogonal basis, all the coefficients of $V_{\ell m}, W_{\ell m}, X_{\ell m}$ must be zero in \eqref{Ode}. Let first  $\ell\geq 1$ and we have six sets of analytical solutions to \eqref{Ode} reading 
\begin{center}
\begin{tabular}{c |c c c |c c c | } 
& (i) & (ii) & (iii) &(iv) &(v)& (vi)\\ 
 \hline
 $f$  & $r^{-\ell-2} $& $- {\ell \over 2\ell+1}(\mu+\lambda)r^{-\ell}$ & $0$ &${2 \over 2l+3} \Big(\mu+{ \ell \over 2\ell+1}(\mu+\lambda)\Big)r^{\ell+1}$ &$0$ &$0$ \\
 $g$ &$0$ & ${2 \over 2\ell -1}\Big(\mu+{\ell +1\over 2\ell+1}(\mu+\lambda)\Big)r^{-\ell}$ &$0$&$ {\ell+1 \over 2\ell+1} (\mu+\lambda)r^{\ell+1}$&$r^{\ell-1}$&$0$ \\
 $h$ &$0$& $0$& $r^{-\ell-1}$ &$0$&$0$&$r^\ell$
\end{tabular}
\end{center}
in which $(i),(ii),(iii)$ are admissible only  for $|r|>0$ while $(iv),(v),(vi)$ are unbounded when $|r|\to \infty$.  Now, we consider the exterior and the interior of the unit sphere separately  in  which we aim to get $\SL V_{\ell m}$. 
For the sake of simplicity, write 
\[
\begin{array}{l}
p_{12}= -{\ell\over 2\ell+1}(\mu+\lambda), \quad p_{22}= {2 \over 2\ell-1}\Big(\mu+{\ell+1\over 2\ell+1} (\mu+\lambda)\Big), \\ q_{11}={2 \over 2l+3} \Big(\mu+{ \ell\over 2\ell+1}(\mu+\lambda)\Big)r^{\ell+1}, \quad q_{21}= {\ell+1 \over 2\ell+1} (\mu+\lambda). 
\end{array}
\]
Write  $\SL  V_{\ell m} $ as a linear combination of the three solutions in the exterior and three in the interior of the unit sphere:  
\bel{v-in-out}
 \SL  V_{\ell m} = \begin{cases}
a_{\ell m}^{in}q_{11}r^{\ell+1} V_{\ell m}+ \Big(a_{\ell m}^{in}q_{21}r^{\ell+1}+b^{in}_{\ell m} r^{\ell-1}\Big) W_{\ell m}+c^{in}_{\ell m} r^\ell X_{\ell m} & |r|<1,\\
\Big(a^{out}_{\ell m} r^{-\ell-2} + b^{out}_{\ell m}p_{12}r^{-\ell} \Big)V_{\ell m}+ b^{out}_{\ell m} p_{22} r^{-\ell} W_{\ell m}+c^{out}_{\ell m} r^{-\ell-1} X_{\ell m}& |r|>1, 
\end{cases}
\ee
where $a^{in/out}_{\ell m},b^{in/out}_{\ell m},c^{in/out}_{\ell m}\in \mathbb R$ are constants to be determined.  In order to determine these six unknown constants, we use the jump relation given by Theorems~\ref{Th:jump-relations}: 
\bel{two}
\llbracket \SL  V_{\ell m} \rrbracket= 0, \quad \llbracket \mathcal T  \SL  V_{\ell m}\rrbracket= V_{\ell m}. 
\ee
Since $V_{\ell m}, W_{\ell m},X_{\ell m}$ is an orthogonal basis, it follows immediately from the first equality in \eqref{two} that 
\[ 
\aligned 
a_{\ell m}^{in}q_{11}=  a^{out}_{\ell m}+ b^{out}_{\ell m} p_{12}
 \quad  a_{\ell m}^{in}q_{21}+b^{in}_{\ell m} = b^{out}_{\ell m} p_{22},
\quad c^{in}_{\ell m}  = c^{out}_{\ell m}. 
\endaligned
\] 
Now take the inner product of the second equation \eqref{two} with $V_{\ell m}, W_{\ell m},X_{\ell m}$. With the computations in \eqref{Tfgh}, we have 
\begin{align*}
2\ell -1=& \mu  a_{\ell m}^{in}q_{11} (3\ell+2)(\ell+1)+ \lambda _{\ell m}^{in}q_{11} (\ell+1)^2 - (\mu+\lambda)\Big(a_{\ell m}^{in} q_{21}  \ell(\ell+1)  + b_{\ell m}^{in} \ell(\ell-1)\Big) \\
&- \mu \Big(a^{out}_{\ell m} (3\ell+2)(-\ell-2)- b^{out}_{\ell m}p_{12}(3\ell+2) \ell\Big) \\&- \lambda \Big((\ell+1)(\ell-2)a^{out}_{\ell m}-b^{out}_{\ell m}p_{12} (\ell+1) l \Big)-(\mu+\lambda)b^{out}_{\ell m}p_{22}\ell^2, 
 \\ 
 0=& \mu \Big(  a_{\ell m}^{in} q_{21}  (3\ell+1) (\ell+1)+  b_{\ell m}^{in} (3\ell+1) (\ell-1)\Big) 
+ \lambda \Big(a_{\ell m}^{in} q_{21} \ell(\ell+1)+ b^{in}_{\ell m} \ell(\ell-1)\Big)\\& -(\mu+\lambda) a_{\ell m}^{in}q_{11} (\ell+1)^2+\mu b^{out}_{\ell m} p_{22} \ell(3\ell+1)+\lambda b^{out}_{\ell m} p_{22} \ell^2 \\&-( \mu+\lambda) \Big(a^{out}_{\ell m} (\ell+1)(\ell+2)+b^{out}_{\ell m}p_{12}(\ell+\ell) \ell\Big),
\\
 0=&  \mu (\ell c^{in}_{\ell m}  +(\ell +1)c^{out}_{\ell m}).
\end{align*}
Hence, we conclude 
\[
\aligned 
 &a_{\ell m}^{in} = {1\over 2 \mu (2\mu+ \lambda )},  \qquad a_{\ell m}^{out}={1 \over 2\ell+3}  \Big(\mu+{ \ell \over 2\ell+1}(\mu+\lambda)\Big) \big( \mu(2\mu + \lambda)\big)^{-1}, \\ &b_{\ell m}^{in} = - {\ell+1 \over 2\ell+1} (\mu+\lambda)  \big(2 \mu(2\mu + \lambda)\big)^{-1},\qquad  
b^{out}_{\ell m}=c^{in}_{\ell m} = c^{out}_{\ell m}=0. 
\endaligned 
\]
Similar computations following the same logic give the results for $W_{\ell m}, X_{\ell m}$ for $\ell \geq 1$. 

In the case where $\ell=0$, we only have to treat $V_{00}$ since $W_{00}=X_{00}=0$.  Using \eqref{v-in-out}--\eqref{two}, we have 
 \bel{l0single}
\SL  V_{00}= \begin{cases} { 1\over 3(2\mu+ \lambda)} r^2  V_{00} & r\leq 1,\\
 { 1\over 3(2\mu+ \lambda)} r^{-1}  V_{00}& r>1. 
\end{cases}
\ee
Notice that the result  \eqref{l0single} is  indeed consistent with the cases where $\ell \geq 1$. 
We have proved therefore the theorem.
\end{proof}
Corollary~\ref{corollaryH&N} can now be deduced. 
\begin{proof}[Proof of Corollary~\ref{corollaryH&N}]
 As is shown in \eqref{boundary-neuman}, we have
\[
 2\DLO  ^*V_{\ell m} =  \mathcal T_{\bf n} ^+ \SL V_{\ell m} +  \mathcal T_{\bf n} ^- \SL  V_{\ell m}. 
\] 
Again, apply the computation in \eqref{Tfgh},  we have 
\[
  \mathcal T^-_{\bf n} \SL V_{\ell m}+ \mathcal T_{\bf n}^+ \SL V_{\ell m}= {-2(2\ell^2+6\ell+1)\mu + 3 \lambda\over (2\ell+1)(2\ell +3) (2\mu+\lambda)}V_{\ell m}. 
\]
Hence, we have \[
\DLO^*V_{\ell m}={2(2\ell^2+6\ell+1)\mu + 3 \lambda\over 2 (2\ell+1)(2\ell+3) (2\mu+\lambda)}. 
\] 
Similar computations give the results for the other two components $W_{\ell m}, X_{\ell m}$. Further, Lemma~\ref{stars} provides the result for the double layer operator $\DLO$.  
\end{proof}
Finally, the proof of Theorem~\ref{H&N'} follows the same structure to that of Theorem~\ref{H&N} in employing the jump relations 
\[
\llbracket \DL \phi \rrbracket = -\phi, \quad \llbracket \mathcal T \DL \phi \rrbracket =0
\]
for the choices $\phi =V_{\ell m}, W_{\ell m}, X_{\ell m} $ respectively. 

\section{Application}
\label{sec:case-study}
We study here a case of an elasticity problem involving several spherical inclusions as an application of the results in the above sections, derive an integral equation formulation and propose a Galerkin formulation thereof based on the vectorial spherical harmonics.

\subsection{Problem setting}
Set  the sets of indices $J_1,J_2,J$ such that $M+1\in J_2$, $ J_1 \cap J_2=\emptyset$ and
 \bel{J1&J2}
 J_1\cup J_2 =\{1,2,\dots, M,M+1\}, \qquad J=  J_1\cup J_2 \textbackslash \{M+1\}, 
 \ee
and let $\Omega_i\subset \mathbb R^3$, $i \in J$ be non-overlapping balls, centered at $x_i
\in \mathbb R^3$ with radius $r_i$, all contained in an additional ball $B_R$  centred at the origin with the radius $R$.   
 Moreover,  define the domains
 \bel{Omegas}
 \Omega_{M+1} = \mathbb R^3\textbackslash B_R, \quad \Omega_0 := B_R \textbackslash  \bigcup\limits_{i\in J_1\cup J_2} \overline \Omega_i, \quad \Omega:= B_R \textbackslash  \bigcup\limits_{i\in J_2} \overline \Omega_i.
\ee 
Denote the boundaries $\Gamma_i= \del \Omega_i$, $ i\in \{0\} \cup J_1\cup J_2$. 
Then, it holds that $\Gamma_0 =\bigcup\limits_{i\in J_1\cup J_2} \Gamma_i$. 
Set further ${\bf n}_0$ as the outward pointing unit normal vector with respect to the domain $\Omega_0$ and ${\bf n}_i$ the outward pointing normal vector with respect to each domain  $\Omega_i$, $i\in J_1\cup J_2$. Then it holds that ${\bf n}_i=- {\bf n}_0$. 
 We refer to Figure~\ref{Geometry configuration} for an illustration of geometry configuration. 

\begin{figure}
\centering
\includegraphics[height = 2.5in]{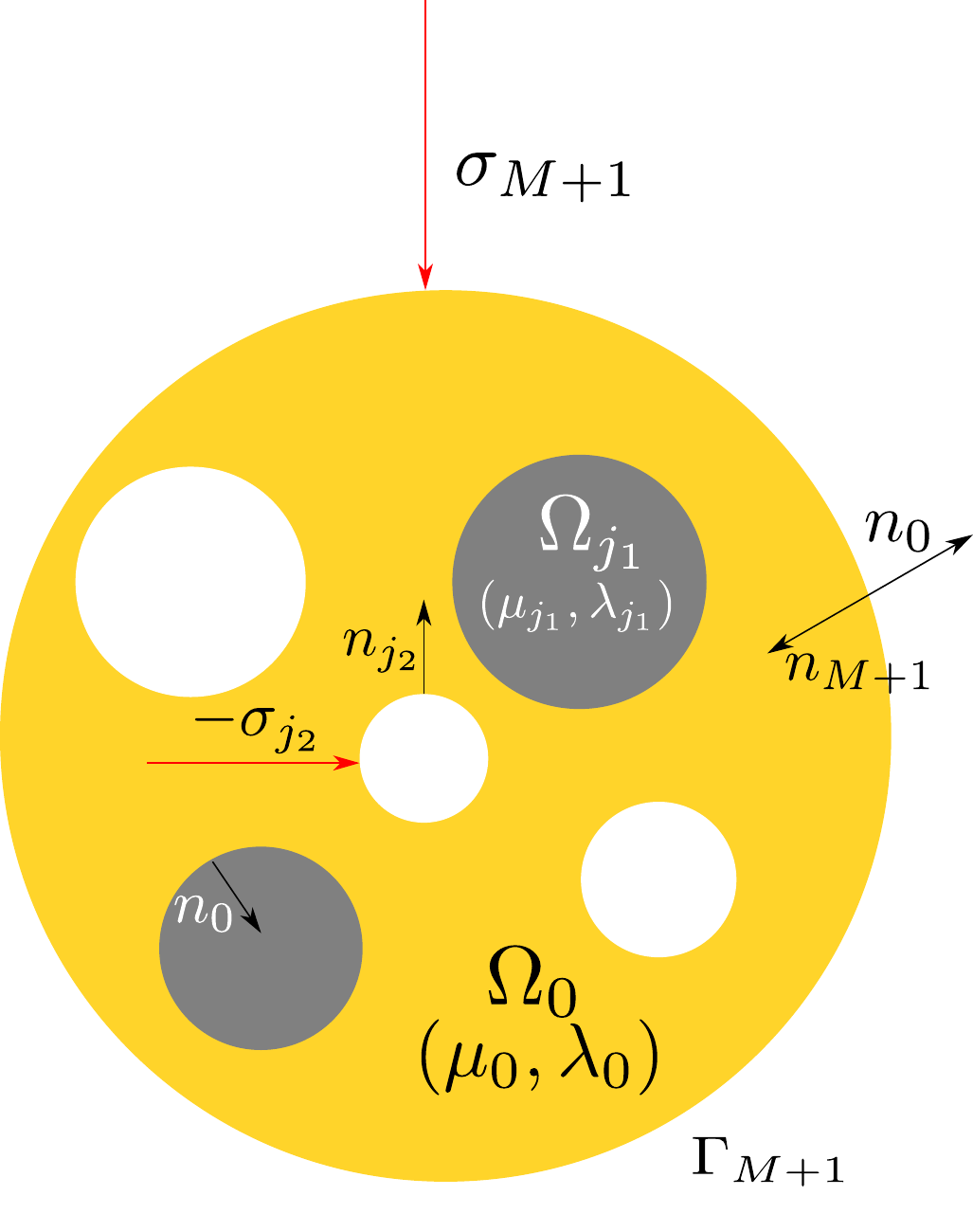} 
\caption{Geometry setting of the model where  $j_1\in J_1$ and $j_2\in J_2$. }\label{Geometry configuration}
\end{figure}

In the numerical example presented below, we assume that each inclusion $\Omega_i$, $i\in  J_1$ is filled with an isotropic elastic medium associated with Lam\'e parameters $\mu_i, \lambda_i$.  The remaining background domain $\Omega_0$ is filled with medium of  Lam\'e constants $\mu_0, \lambda_0$. Further, we denote  $
{\mathcal T^i_{{\bf n}_j}}^{\pm}$  the normal derivative operator acting on the boundary $\Gamma_j$ with Lam\'e constants $\mu_j, \lambda_j$ and the normal vector ${\bf n}_j $:
\[
{\mathcal T^i_{{\bf n}_j}}^{\pm} u = \gamma_j^\pm\Big( (2\mu_i e(u)+ \lambda_i \mathop {\rm Tr}e(u) \mathop {\rm Id}){\bf n}_j\Big)
\] 
where $\gamma^\pm_j$ are the exterior and interior the trace operators on $\Gamma_j$,  following the notations given by \eqref{trace} by taking $\Omega^- = \Omega_j$.  
Define the parameter 
 $s_i$ by 
\bel{si}
s_i =\begin{cases}
-1& i=M+1, \\
1& \mbox{else}. 
\end{cases}
\ee
 In particular, write 
\bel{s-jump}
\llbracket \mathcal T u \rrbracket =  {\mathcal T^0_{s_i{\bf n}_i}}^- u -  {\mathcal T^0_{s_i {\bf n}_i}}^+u, \qquad \mbox{on}~\Gamma_ i, ~i\in J_1\cup J_2
\ee
and it is obvious that 
\[
 {\mathcal T^0_{s_i{\bf n}_i}}^\pm = s_i {\mathcal T^0_{{\bf n}_i}}^\pm. 
\]
For given $f_i\in \Hnh(\Gamma_i)^3$, we impose the transmission condition:
\[
\llbracket u \rrbracket= 0, \qquad\llbracket \mathcal T u \rrbracket = f_i, \qquad \mbox{on}~\Gamma_i, i\in J_1. 
\]
Further, we let the domain $\Omega$ be subjected  of a given stress tensor on its boundary. Indeed, we let each part of the boundary   $\Gamma_i$, $i\in J_2$ be  subjected to a  given stress tensor $\sigma_i \in H ^{1/2}(\Gamma_i)^3$:
\bel{stress-relation}
	 {\mathcal T^0_{{\bf n}_i}}^+ u=-s_i\sigma_i. 
\ee

Then, we consider the solution $u\in \Vn{\Omega}$ to the following interface problem: 
\bel{casestudy}
\begin{array}{rlll}
-\mathop {\rm div}\big(2\mu_i e(u)+ \lambda_i \mathop {\rm Tr }e(u) \mathop {\rm Id}\big) &\hspace{-7pt}= 0,  &\mbox{in } \Omega_i,~i\in \{0\}\cup J_1, \\
\llbracket u \rrbracket &\hspace{-7pt} =0, &  \mbox{on }\Gamma_i~i\in J_1, \\
\llbracket \mathcal T u \rrbracket &\hspace{-7pt} =f_i, &  \mbox{on }\Gamma_i~i\in J_1, \\
 {\mathcal T^0_{{\bf n}_i}}^+ u  &\hspace{-7pt}=-s_i\sigma_i,  &\mbox {on } \Gamma_i, i\in J_2.
\end{array}
\ee
Standard arguments involving the Lax-Milgram theorem yields the well-posedness of the problem. 
\begin{remark}
In our case, we pre-defined the index set $J_2$ with the condition  $M+1\in J_2$. However, by equipping the domain $\Omega_{M+1}$ (which is indeed the exterior domain of the sphere $B_R$) with Lam\'e parameters $\mu_{M+1}, \lambda_{M+1}$,  one can also relax the setting and consider the case where $M+1\in J_1$. 
\end{remark}

\subsection{Integral equation}
Let $u\in \Vn{\Omega}$ denote the solution to \eqref{casestudy} satisfying $\llbracket u \rrbracket=0$ on $\Gamma_i$ for $i\in J_1$ and define $\nu$ on $\Gamma_0$ by
\[
	\nu= \gamma_0^- u
	\qquad\mbox{as well as}\qquad
	\quad \nu_i= \nu|_{\Gamma_i}, \quad \forall i\in J_1\cup J_2 ,
\]  
where $\gamma_0^-$ is the trace operator $\gamma_0^- \colon H^1(\Omega_0)^3 \rightarrow \Hh(\Gamma_0)^3$  defined by \eqref{trace} for $\Omega^-=\Omega_0$.  
To deduce an integral equation for $\nu$, we first introduce an auxiliary problem: find a solution 
$v\in H^1(\Omega )^3$
to \bel{auxiliary}
\begin{array}{rll}
-\mathop {\rm div}\big(2\mu_0 e(v)+ \lambda_0  \mathop {\rm Tr }e(v) \mathop {\rm Id}\big) &\hspace{-7pt} = 0, ~&\mbox{in } \Omega \textbackslash \Gamma_0, 
\\
\gamma_0^-  v&\hspace{-7pt}= \gamma_0^-  u, ~&\mbox {on }\Gamma_0. 
\end{array}
\ee
The auxiliary  problem \eqref{auxiliary} admits a unique solution  $v$ in $\Omega$ and observe that
\[
v=u, \quad~\mbox{in}~\Omega_0,
\] 
but is different in all $\Omega_i$, $i\in J_1$.
Further,  there exists a global density $\phi$  supported on $\Gamma_0$ such that 
\bel{globle-present}
\begin{array}{ll}
v=\SL_ G^0 \phi=  \sum\limits_{i\in J_1\cup J_2} \SL^0_i \phi_i \quad &\text{in}~\Omega_0 \\
v=\SLO_ G^0 \phi=   \sum\limits_{i\in J_1\cup J_2} \SLO^0_i \phi_i \quad & \mbox{on}~ \Gamma_0, 
\end{array}
\ee
where $\SL_G^0$ is the global layer potential \eqref{single-layer} with Lam\'e parameters $\mu_0,\lambda_0$ defined on the whole boundary $\Gamma_0$ while $\SL_i^0$ is the local single layer potential with Lam\'e parameters $\mu_0,\lambda_0$ defined locally on the sphere $\Gamma_i$ and $\SLO_ G^0,  \SLO^0_i$ their corresponding single layer boundary operators \eqref{single-layer-boundary}. Further, according to Theorem~\ref{Th:jump-relations}, the density $\phi_i$ is given by  the jump relation 
\bel{phiSm}
\phi_i=\llbracket  \mathcal T  v  \rrbracket =\Big( {\mathcal T_{s_i{\bf n}_i}^0}^- v -{\mathcal T_{s_i{\bf n}_i}^0 }^+ v\Big)= s_i\Big(  {\mathcal T_{{\bf n}_i}^0 }^- u  -{\mathcal T_{{\bf n}_i}^0}^+ v\Big), \quad~\mbox{on}~\Gamma_i,  i\in J_1\cup J_2.  \ee
The last equivalence in \eqref{phiSm} is obtained because  $u=v$ on $\Omega_0$. 
Further, both solutions $u, v$ can be represented by some local densities in each domain $\Omega_i$: 
\[
\begin{array}{ll}
u|_{\Omega_i}=\SL_i \varphi_i, \quad & i\in J_1 \\
v|_{\Omega_i}=\SL_i^0 \psi_i\quad & i\in J_1\cup J_2,
\end{array}\]
where $\SL_i$ is the local layer potential with Lam\'e parameters $\mu_i,\lambda_i$ while $\SL_i^0$ with $\mu_0,\lambda_0$, both of which are defined locally on the sphere $\Gamma_i$. For the corresponding  single layer boundary operators $\SLO_i^0, \SLO_i$ defined by \eqref{single-layer-boundary}, we have 
\[
\SLO_i^0 \psi_i= \SLO_i \varphi_i = \nu_i, \quad i\in J_1. 
\] 
According to Corollary~\ref{Corol:invertible}, the single layer boundary operators $\SLO_i^0, \SLO_i$ are invertible, so we have  
\bel{inverse-values}
 \begin{array}{ll}
\psi_i= {\SLO_i^0}^{-1} (\gamma_i v)=    {\SLO_i^0}^{-1} \nu_i,  \quad & \mbox{on }~\Gamma_i, i\in J_1\cup J_2, 
\\ 
\varphi_i= \SLO_i^{-1}(\gamma_i u ) = \SLO_i^{-1}\nu_i,\quad & \mbox{on }~\Gamma_i, i\in J_1.
 \end{array}
\ee 
Now, according to Theorem~\ref{double-in/out}, we have 
 \bel{double-local}
 \begin{array}{ll}
{\mathcal T_{{\bf n}_i}^0}^- \SL_i^0  \psi_i =   s_i{1\over 2 } \psi_i  +  {\DLO_i^0}^*\psi_i,  \quad  &\mbox{on }~\Gamma_i, i\in J_1 \cup J_2, \\
{\mathcal T_{{\bf n}_i}^i}^-  \SL_i \varphi_i = s_i  {1\over 2 } \varphi_i+   \mathcal \DLO^*_i\varphi_i, \quad & \mbox{on }~\Gamma_i, i\in J_1, \\
 \end{array}
 \ee
 where $ {\DLO_i^0}^*$, $ \DLO_i^*$ are adjoint double layer boundary operators \eqref{double-boundary*} defined locally on $\Gamma_i$ with Lam\'e constants $\mu_0, \lambda_0$ and $\mu_i, \lambda_i$ respectively.  In problem \eqref{casestudy}, we have 
 \[
 {\mathcal T_{{\bf n}_i}^0 }^+   u=  \begin{cases}
 {\mathcal T_{{\bf n}_i}^i}^-  \SL_i  \varphi_i-f_i, & \mbox{on }~\Gamma_i, i\in J_1, \\
 - s_i\sigma_i, & \mbox{on }~\Gamma_i, i \in J_2, 
 \end{cases}
  \]
  where $f_i\in \Hnh(\Gamma_i)^3,\sigma_i\in \Hh(\Gamma_i)^3$ are given transmission conditions and Neuman boundary conditions respectively. Providing \eqref{phiSm}--\eqref{double-local}, we have 
  \[
  \phi_i=\begin{cases}
  {1\over 2} ({\SLO_i^0}^{-1}- \SLO_i^{-1}) \nu_i +  ({\DLO_i^0}^*{\SLO_i^0}^{-1}-{\DLO_i}^* \SLO_i^{-1}) \nu_i+f_i, &i\in J_1, \\
{1\over 2}  {\SLO_i^0}^{-1} \nu_i +s_i {\DLO_i^0}^*{\SLO_i^0}^{-1}\nu_i+s_i\sigma_i,&i \in J_2.
  \end{cases}
  \]
 According to \eqref{globle-present}, we have 
 \bel{integral-equation}
 (\mathop{\rm Id} - \SLO_G^0\mathcal L)\nu =  \SLO_G^0\Sigma
 \ee
 where $\mathcal L$ is an operator defined on each $\Gamma_i$: 
 \begin{align*}
\mathcal L|_{\Gamma_i} &= \mathcal L_i  = {1\over 2} ({\SLO_i^0}^{-1}- \SLO_i^{-1}) +({\DLO_i^0}^*{\SLO_i^0}^{-1}-{\DLO_i}^* \SLO_i^{-1}), & i\in J_1,
\\
\mathcal L|_{\Gamma_i} &=\mathcal L_i  = {1\over 2}  {\SLO_i^0}^{-1} \nu_i +s_i {\DLO_i^0}^*{\SLO_i^0}^{-1}\nu_i ,  & i\in J_2, 
 \end{align*}
  and the  vector $\Sigma$ satisfies 
 \[
 \Sigma|_{ \Gamma_ i}  =\begin{cases}
f_i,& i\in J_1,\\
 \sigma_i,& i \in J_2. 
 \end{cases}
 \]

\subsection{Galerkin approximation}
Introduce $V_{N,i}, i\in J_1\cup J_2$   the set spanned by vector spherical harmonics \eqref{vector-harmonics} on the sphere $\Gamma_i$ with a maximum degree $N$: 
\bel{vl}
V_{N,i}=\Big\{\sum\limits_{\ell=0}^N \sum\limits_{m=-\ell}^\ell  \sum\limits_{k=1}^3 [y_i]_{\ell m}^k Y_{\ell m}^{ki}(x)  \Big| [y_i]_{\ell m}^k\in \mathbb R\Big\},  
\ee
where we write $Y_{\ell m}^1=V_{\ell m}, Y_{\ell m}^2=W_{\ell m}, Y_{\ell m}^3=X_{\ell m} $  and 
$
Y_{\ell m}^{ki}(x)= Y_{\ell m}^k \left({x-x_i\over r_i}\right). 
$
Define also the global set
\bel{vl-global}
V_{N}=\bigotimes\limits_{i\in J_1\cup J_2} V_{N,i}.
\ee
We look for the approximation of  $\nu_N\in V_N$ to \eqref{casestudy} with 
\bel{Galerkin}
\forall i\in J_1\cup J_2,  \forall v_{N,i}\in V_{N,i}: \quad  \langle \nu_{N,i}-\mathcal V^0_G\mathcal L \nu_{N},v_{N,i}\rangle_{\Gamma_i} =  \langle \mathcal V^0_G \Sigma, v_{N,i}\rangle_{\Gamma_i}.
\ee
In practice, we use the quadrature \eqref{quadrature} to approximate the inner product and the approximate solution $\nu_N$  on each $\Gamma_i$ thus satisfies 
\[
\forall i\in J_1\cup J_2,  \forall v_{N,i}\in V_{N,i}: \quad  \langle \nu_{N,i}-\mathcal V^0_G\mathcal L \nu_{N},v_{N,i}\rangle_{\Gamma_i,t} =  \langle \mathcal V^0_G\Sigma, v_{N,i} \rangle_{\Gamma_i,t}. 
\]
Denote by $\mathcal M= 3(N+1)^2 (|J_1|+|J_2|)$ the number of degrees of freedom. The $\mathbb R^{\mathcal M}$-vector  collecting all the coefficients $[y]_{\ell m}^k$  denoted by ${\bf \Lambda}$ yields the linear system
\bel{linear}
  ({\bf D}- {\bf N}){\bf \Lambda} ={\bf F},
\ee
where by \eqref{orthogonal}, \eqref{inner-product}, the $\mathcal M \times \mathcal M$ diagonal matrix ${\bf D}$ is given by
\[
[D_{ii}]_{\ell m,\ell m}^{11} =(2\ell+1)(\ell+1), 
\qquad
[D_{ii}]_{\ell m,\ell m}^{22} =(2\ell+1)\ell, 
\qquad
[D_{ii}]_{\ell m,\ell m}^{33} =(\ell+1)\ell, 
\] 
and where $\bf  {N}$ is a $\mathcal M\times \mathcal M $ matrix with coefficients
\bel{N}
[N_{ij}]_{\ell m,l'm'}^{kk'}=\big\langle\mathcal V_G^0 \mathcal L_jY_{\ell'm'}^{k'},  Y_{\ell m}^{ki}\big\rangle_{\Gamma_i,t}.
\ee
The right hand side  ${\bf F}\in \mathbb R^{\mathcal M}$ is given by its coefficients:
\bel{bfF}
[F_i]_{\ell m}^k=\langle \mathcal V_G^0
\Sigma,Y_{\ell m}^{ki} \rangle_{\Gamma_i,t}. 
\ee
To derive the entries of the matrix \eqref{N}, recall the spectral results in Lemma~\ref{H&N} so that we have on $\Gamma_{j}$
\begin{align*}
{\SLO_j}^{-1}Y_{\ell'm'}^{k'j}&= {1\over r_j \tau^{k'j}_{\SLO,\ell'}}Y_{\ell'm'}^{k'j}, 
&{\SLO_j^0 }^{-1} Y_{\ell'm'}^{k'j} & ={1\over  r_j\tau^{k'0}_{\SLO,\ell'}}Y_{\ell'm'}^{k'j}, 
\\ 
\DLO^*_j Y_{\ell'm'}^{k'j}&=  \tau^{k'j}_{\DLO^*,\ell'}Y_{\ell'm'}^{k'j}, 
&\DLO^{*0}_j Y_{\ell'm'}^{k'j}&= \tau^{k'0}_{\DLO^*,\ell'}Y_{\ell'm'}^{k'j},
\end{align*}
where $\tau^{k'j}_{\SLO,\ell'}, \tau^{k'j}_{\DLO^*,\ell'}$ are eigenvalues of the single and double layer  boundary operator with Lam\'e constants $\mu_j, \lambda_j$ given by  \eqref{aon}, \eqref{adon} respectively. Hence, we have 
\[
\mathcal L_ j Y_{\ell'm'}^{k'j} = {1\over r_j} C_{j\ell'k'}  Y_{\ell'm'}^{k'j}, 
\]
where the constant $ C_{j\ell'k'}$ reads 
\[
 C_{j\ell'k'}= \begin{cases}
{1\over 2}\Big({1\over  \tau^{k'0}_{\DLO^*,\ell'}} -{ 1\over  \tau^{k'j}_{\SLO,l'} }\Big)+  \Big({\tau^{k'0}_{\DLO^*,\ell'}  \over \tau^{k'0}_{\SLO,\ell'}} -{\tau^{k'j}_{\DLO^*,\ell'}  \over \tau^{k'j}_{\SLO,\ell'} }\Big)& j\in J_1, \\
{1/2 +s_i\tau^{k'0}_{\DLO^*,\ell'} \over \tau^{k'0}_{\SLO,\ell'}}&j\in J_2.
 \end{cases} 
\]
Recall that $s_i$ denotes the parameter defined by \eqref{si}. 
Further, by \eqref{ain}, \eqref{aout}, 
\[
(\SL_j^0Y_{\ell'm'} ^{k'})(x)= \begin{cases}
r_j\Big[\underline {Y_{\ell'm'}}\Big({x-x_j\over|x-x_j|}\Big) A_{\SLO, \ell'}^{in}\Big({x-x_j\over r_j}\Big)\Big]_{k'} & |x-x_j|\leq r_j, \\
r_j\Big[\underline {Y_{\ell'm'}}\Big({x-x_j\over|x-x_j|}\Big)A_{\SLO, \ell'}^{out}\Big({x-x_j\over r_j}\Big) \Big]_{k'} & |x-x_j| > r_j, \end{cases} 
\]
where  $[\cdot]_{k'}$ denotes the $k'$-th column of the obtained matrix.  Hence, the coefficient $[N_{ij}]_{\ell m,\ell'm'}^{kk'}$ of the matrix ${\bf N}$ \eqref{N} reads 
\[
\aligned 
& [N_{ij}]_{\ell m,\ell'm'}^{kk'}=  \langle \mathcal V_G^0 \mathcal L_j Y_{\ell'm'}^{k'}, Y_{\ell m}^{ki}\rangle_{\Gamma_i, t}
\\  
&= C_{j\ell'k'}\sum\limits_{t=1}^{T_g}w_tY_{\ell m}^{ki}(s_t)  \SL_j ^0Y_{\ell'm'}^{k'}(x_i+r_i s_t)
= C_{j\ell'k'}\sum\limits_{t=1}^{T_g}w_t Y_{\ell m}^{ki}(s_t) \Big[\underline {Y_{\ell'm'}}\Big({y_{ij}^t\over|y_{ij}^t|}\Big) A_{\SLO, \ell'}^{f(j)}\Big({y_{ij}^t\over r_j}\Big)\Big]_{k'}, 
\endaligned 
\]
where $y_{ij}^t=x_i+r_is_t-x_j  $ and $f(j)$ takes the value 
\[
f(j) =\begin{cases}
in & j =M+1, \\
out & else. 
\end{cases}
\]
In particular, when $i=j$, we use \eqref{orthogonal} for the exact value of the inner product and obtain: 
\[
[N_{ii}]_{\ell m,\ell'm'}^{kk'}=  \sum\limits_{l'=0}^N\sum\limits_{m'=-\ell'}^{\ell'} C_{i\ell'k'} \tau_ {\SLO,\ell} ^{k',0} \langle Y_{\ell m}^{ki},  Y_{\ell'm'}^{k'i}  \rangle_{\Gamma_i}. 
\]
Finally, the right-hand side vector ${\bf F}$ with coefficient $[F_i]_{\ell m}^k$ is given by:
\[\aligned 
 {[F_i]}_{\ell m}^{k}=&\sum\limits_{\ell'=0}^N\sum\limits_{m'=-\ell'}^{\ell'}\sum\limits_{k' =1}^3\Bigg(   r_i  [\Sigma_i ]_{\ell'm'}^{k'}  \tau^{k'0}_{\SLO,\ell'}  \langle Y_{\ell'm'}^{k'i},  Y_{\ell m}^{ki} \rangle_{\Gamma_i} \\& + \sum\limits_{\mathclap{\substack{j \in J_1\cup J_2\\ j\neq  i }}}r_ j \sum\limits_{t=1}^{T_g} [\Sigma_j]_{\ell'm'}^{k'}w_tY_{\ell m}^{ki}(s_t)  \Big[\underline {Y_{\ell'm'}}\Big({y_{ij}^t\over|y_{ij}^t|}\Big) A_{\SLO, \ell'}^{f(j)}\Big({y_{ij}^t\over r_j}\Big)\Big]_{k'}\Bigg),~i\in J_1\cup J_2
\endaligned 
\]
where $[\Sigma_i]_{\ell m}^k\in \mathbb R$  is determined by the right-hand side vector $\Sigma$ in \eqref {integral-equation} : 
\[
\Sigma|_{\Gamma_i}=\sum\limits_{\ell=0}^N\sum\limits_{m=-\ell}^\ell\sum\limits_{k=1}^3 [\Sigma_i]_{\ell m}^k Y_{\ell m}^{ki}(x). 
\]
\begin{remark}
A similar physical model called ``Finite Cluster Model'' was considered in~\cite{Urklain} in where an algebraic formulation is derived through the use of M2L-operators (using the fast multipole method terminology). 
However, with the jump relations given in~\eqref{jump-relation}, the algebraic formulation of the  ``Finite Cluster Model'' can be proven to be equivalent to the discrete integral formulation~\eqref{Galerkin} presented above. 

It shall be noted that the mathematical framework introduced here through the use of layer potentials and boundary operators in order to derive an integral equation~\eqref{integral-equation} defining the exact solution and and the subsequent introduction of the Galerkin discretization~\eqref{Galerkin} allows a mathematical analysis which will be subject of an upcoming work.
\end{remark}
\section{Numerical tests}
\label{sec:8}

For all following computations, we chose the number of Lebedev integration points $T_g$ such that, for given $N$, products of two scalar spherical harmonics of maximal degree $N$, thus spherical harmonics of degree $2N$, are integrated exactly.  
The number of points can then be extracted from Table~\ref{tab:Lebedev}.

\subsection{One sphere model}
We start with a simple model involving only one single sphere whose solution can be computed analytically in order to assess the convergence  of the method in this simple setting. 
For simplicity,  let $\mathbb S^2 \subset  \mathbb R^3$ be the unit sphere on which a stress tensor $
\sigma\in \Hh(\mathbb S^2)^3$ is imposed and let the Lam\'e constants be $\mu_0=\lambda_0=1$.  Let ${\bf n}$ be the outward pointing normal vector with respect to the unit sphere $\mathbb S^2$. The solution $u\in \Vn{B}$ to  the problem 
\[
\begin{array}{r}
-\mathop {\rm div}\big(2  e(u)+   \mathop {\rm Tr }e(u) \mathop {\rm Id}\big) = 0, \quad \mbox{in }  B , \\
 {\mathcal T_{\bf n} } ^- u  =\sigma \quad \mbox {on } \mathbb S^2, 
\end{array}
\]
reads
\bel{integral-u}
u(x)= \big(\SL (1+\DLO^*)^{-1}\sigma \big)(x), \quad \forall x\in\mathbb R^3, 
\ee
with  $\SL$ being the single layer potential \eqref{single-layer}  and $\DLO^*$ the adjoint of the double layer boundary operator \eqref{double-boundary*}.  For the given tensor $\sigma$, if there exists an integer $\ell_{\mathop {\rm ex}}$ such that we can expand~$\sigma$ by means of vector spherical harmonics up to order $\ell_{\mathop {\rm ex}}$:
\[
\sigma(x) = \sum\limits_{\ell=1}^{\ell_{\mathop {\rm ex}} }\sum\limits_{m=-\ell}^{\ell}\sum\limits_{k=1}^3 [\Sigma]_{\ell m}^k Y_{\ell m}^k(x),
\]
then the exact  solution restricted to the sphere $\Lambda_{\mathop {\rm ex}}  = u|_{\mathbb S^2}$ is  given explicitly by 
\bel{lambda_integral}
\Lambda_{\mathop {\rm ex}} =  \sum\limits_{\ell=1}^{\ell_{\mathop {\rm ex}} }\sum\limits_{m=-\ell}^{\ell}\sum\limits_{k=1}^3 {\tau_{\SLO,\ell}^k \over {1\over 2} +  \tau_{\DLO^*,\ell}^k} [\Sigma]_{\ell m}^k Y_{\ell m}^k(x),
\ee
where $\tau_{\SLO,\ell}^k, \tau_{\DLO^*,\ell}^k$ are the eigenvalues of the single layer boundary operator and the adjoint double layer boundary operator given by \eqref{aon} and \eqref{adon} resp.  They only concern  in computing \eqref{lambda_integral} is that the denominator tends to zero if $\tau_{\DLO^*,\ell}^k$ approaches $-1/2$. Recall that according to Remark~\ref{-1/2}, the only possible eigenvectors of $\DLO^*$ with the eigenvalue $-1/2$ are $W_{1,-1}, W_{1,0}, W_{1,1}$. According to Appendix A, we see that they are constant and parallel to the cartesian basis ${\bf e}_i, i=1,2,3$.  To ensure that \eqref{lambda_integral} is well defined and these modes avoided, we simply impose that 
\[
\int_{\mathbb S^2} \sigma=0. 
\] 
We consider the following four cases: 
\bei
\item  Case 1. $\sigma = -(x, y, z)^\top$. 
\item Case 2. $\sigma = -(x,0,0)^\top$. 
\item Case 3. $\sigma = -(x^7, y^7, z^7)^\top$. 
\item Case 4. $\sigma = -\big(\sin (2\pi x), \sin (2\pi y), \sin (2\pi z)\big)^\top$.
\eei

Table~\ref{one-sphere-model} lists the $L^2$ norm of the numerical solution on the unit sphere $||\Lambda_s||_{L^2}$ in each case with different degrees of vector spherical harmonics and the relative error is defined by 
\bel{Re}
{\tt Re} ={||\Lambda_s- \Lambda_{{\rm ex}}||_{L^2} \over  ||\Lambda_{{\rm ex}}||_{L^2}}. 
\ee

 The exact solutions  $\Lambda_{{\rm ex}}$  in the first three cases are exactly computed by the  \eqref{integral-u}, \eqref{lambda_integral}  while the in the last cases, the ``exact'' solution is obtained by taking a large enough $\ell_{\mathop {\rm ex}}$ (in this case $\ell_{\mathop {\rm ex}}$ = 50). 

\begin{table*}
 \centering
\ra{0.8}
\scalebox{0.7}{
\begin{tabular}{@{}crrrrrrrr@{}}\toprule
$N$ & Case 1 && Case 2 && Case 3 && Case 4
\\
\midrule
2 & 0 & & 0  & & 1.985e-01& & 5.375e-01 &\\
5 & 0 & & 0  & & 4.020e-03& & 6.797e-02 &\\
8 & 0 & & 0  & & 4.796e-09& & 1.370e-04 &\\
11 & 0 & & 0 & & 0 & & 7.892e-13&\\
14 & 0 & & 0 & & 0 & &7.097e-13&\\
\bottomrule
\end{tabular}}
\caption{Relative error ${\tt Re}$ of the approximation to the one-sphere model. \label{one-sphere-model}}
\end{table*}

\subsection{Convergence with respect to the degree $N$}
We study now the convergence of the error measured in the $L^2$ norm with respect to the degree $N$ of the vector spherical harmonics. 
Using the notation introduced in Section~\ref{sec:case-study}, we test a model with $M=2$ and 
\[
	J_1=\{1\}, 	\quad J_2=\{2,3\}. 
\]
Let $\Gamma_1$ be the  sphere  centered at $( 1, 0, 0)$ respectively with radius $0.1$ and $\Gamma_2$  centered at $(-1,0,0)$ with radius $0.1$ while $\Gamma_3$ is centered at $(0,0,0)$ with radius $2$. 
The  inclusion $\Omega_1$  is filled with a medium represented by the Lam\'e constants $\mu_1=10, \lambda_1=10$  while the background domain $\Omega_0$  uses $\mu_0=1, \lambda_0=1$ as Lam\'e parameters. 

The interface condition $\llbracket \mathcal T u \rrbracket= 0$ is imposed on $\Gamma_1$ while the spheres  $\Gamma_2, \Gamma_3$  are subjected to a stress tensor $\mathcal T_{{\bf n}_2}^+ u= -\sigma_2$ respectively $\mathcal T_{{\bf n}_3}^+ u=\sigma_3$.  Table~\ref{Geo-co-N}   illustrates the parameters of the above geometry configuration.

\begin{table}
\centering

\scalebox{0.7}{
\begin{tabular}{ ccccccccc}
Set &Sphere &Center &Radius &Lam\'e constants & Stress tensor & Transmission
\\  \hline
$J_1$&$\Gamma_1  $ &   $(1,0,0)$  &0.1 & $\mu_1=10, \lambda_1=10$&---------& $\llbracket \mathcal T u \rrbracket= 0$  \\ 
$J_2$&$\Gamma_2  $ &   $(-1,0,0)$  &0.1 &--------- &$\mathcal T_{{\bf n}_2}^+ u = -\sigma_2$&--------- \\ 
$J_2$&$\Gamma_3  $ &   $(0,0,0)$  &2  &$\mu_0=1, \lambda_0=1$ &  $\mathcal T_{{\bf n}_3}^+ u = \sigma_3$&---------\\ 
 \hline
\end{tabular}
}
\caption{Geometric configuration of the case study for convergence with respect to the degree of vector spherical harmonics $N$ involving three spheres. \label{Geo-co-N} } 
\end{table}

We now test two cases to see the relation between the relative error and the degree of the spherical harmonics for different kinds of imposed stress tensors  $\sigma_2,\sigma_3$: 
\begin{enumerate}
 \item The two stress tensors are set to be smooth functions such that 
\begin{equation}
\label{eq:TC1}
 \begin{array}{l}
 \mathcal T_{{\bf n}_2}^+ u
  = -\sigma_2 
  = 10\big(\sin (2\pi (x+1)), \sin(2\pi (y+1)),  \sin(2\pi (z+1)) \big )^\top, 
 \\ 
 \mathcal T_{{\bf n}_3}^+ u
 = \sigma_3 
 = -2 \big(\sin (2\pi x), \sin(2\pi y),  \sin(2\pi z)\big)^\top.
\end{array}
\end{equation}
\item The stress tensors are set to be piecewise smooth such that 
\begin{equation}
\label{eq:TC2a}
\mathcal T_{{\bf n}_2}^+ u 
= - \sigma_2
=
\begin{cases} 
(0.2,0,0)^\top &x\geq 0, \\
-(0.2,0,0)^\top &x<0,
\end{cases}
\end{equation}

and 
\begin{equation}
\label{eq:TC2b}
\mathcal T_{{\bf n}_3}^+ u 
= \sigma_3
=
\begin{cases} 
(1,0,0)^\top &x\geq 0, \\
-(1,0,0)^\top &x<0.
\end{cases}
\end{equation}
\end{enumerate}

We compute the ``exact'' solution to the problem with a large degree of vector spherical harmonics ($N_{\rm ex}=50$) for both cases. 
In the Figure~\ref{convergence}, we illustrate  the log of the relative error \eqref{Re} with respect to the degree $N$ of spherical harmonics of the two tests above. 
We observe exponential convergence in the first case where the given stress tensor is regular. 
In the second case, the situation is less clear as an initial pre-asymptotic is followed by a very fast convergence and the asymptotic regime is not yet reached, but the absolute error is already very small.

\begin{figure}
\begin{minipage}[t]{0.2\textwidth}\centering
\includegraphics[trim=1.5in 3in 6in 3in,width=0.25\textwidth]{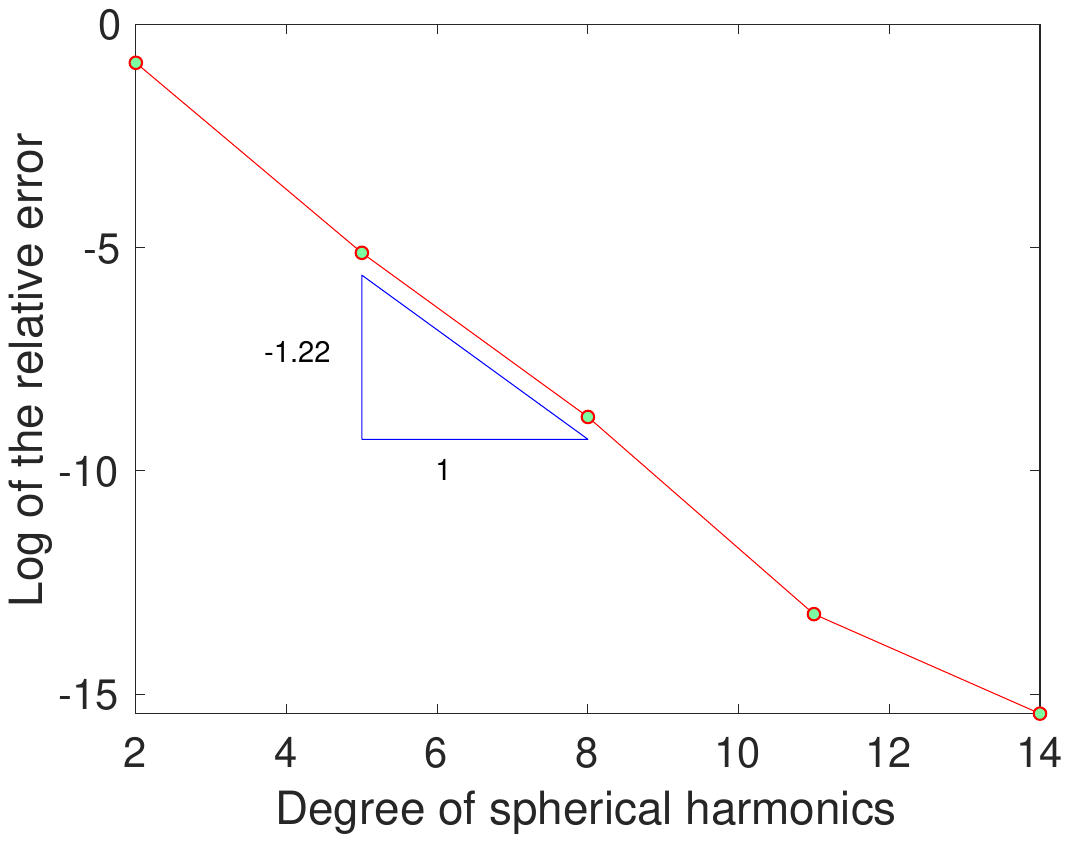}
\end{minipage}\hspace{0.28\textwidth}
\begin{minipage}[t]{0.2\textwidth}\centering
\includegraphics[trim=1.5in 2.5in 6in 3in,width=0.32\textwidth]{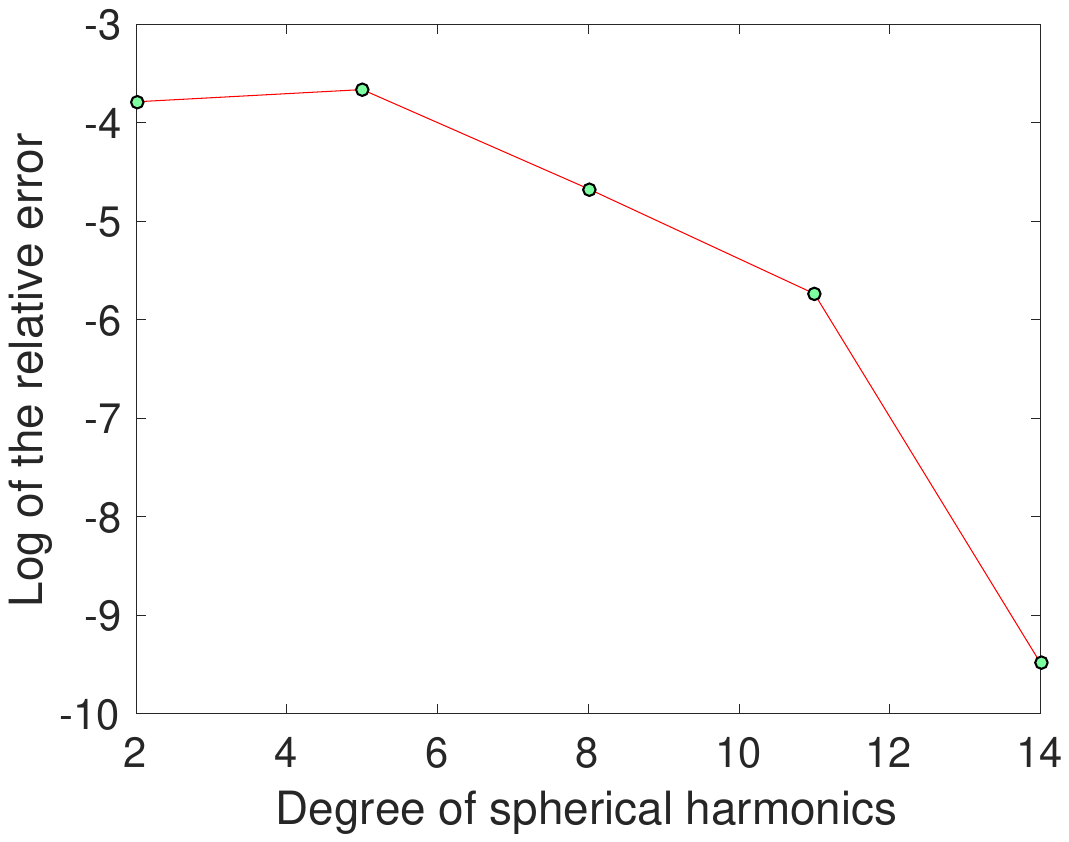}
\end{minipage}
\caption{
The $L^2$ error of the approximation with respect to the degree of spherical harmonics for the test cases \eqref{eq:TC1} (left) and \eqref{eq:TC2a}--\eqref{eq:TC2b} (right).
\label{convergence}}
\end{figure}

\subsection{Computational cost}
\label{sec:Cost}
Next, we study the computational cost of our numerical method by  considering an ``embedded model'' with $M$  inclusions by increasing the value of  $M$.  We do the following test with Matlab on an iMac with a 2,7 GHz Intel Core i5 processor. 

We consider a case  where  a stress tensor $  {\mathcal T^0{{\bf n}_{M+1}}}^+ u  = -{1\over R} (x,y,z)^\top$ is imposed  on a origin-centered  sphere   with radius $R$, denoted by $\mathbb S^2_R$.  Inclusions are taken to be all the spheres  with radii $0.1$, centered on a cubic lattice $\mathbb Z^3$ and which are contained in $\mathbb S^2_R$. We increase the number of inclusions $M$ by increasing the value of the radius $R$ of the big sphere. Table~\ref{number-radius} lists the number of spheres with respect to the radius $R$ that grows of course cubically. 

\begin{table}
\centering
\scalebox{1}{
\begin{tabular}{ccccccc}
 \hline
Radius of the big sphere & 1  & 2&3 & 4&5 \\ 
Number of total spheres & 2  & 28&94 & 252&486\\
 \hline
\end{tabular}}
\caption{Number of spheres w.r.t the radius $R$\label{number-radius} } 
\end{table}

 We fill  each small inclusion with a medium associated with the Lam\'e constants $\mu_i=10, \lambda_i=10$, $i=1,..., M$ and take the transmission condition $\llbracket \mathcal T u\rrbracket=0$ on each embedded sphere. Further, the Lam\'e constants of the background domain are fixed to be $\mu_0=1, \lambda_0=1$.   The degree of the vector spherical harmonics is chosen to be $N=3$. Further, we stop the iterative solver of the linear system when the residual is smaller than $10^{-6}$. Figure~\ref{illustration-2} illustrates the computed elastic deformation of the model computed when $R=3$.   The colorcode represents the modulus of the displacement. 
 \begin{figure}
\centering 
\includegraphics[height = 2.5in]{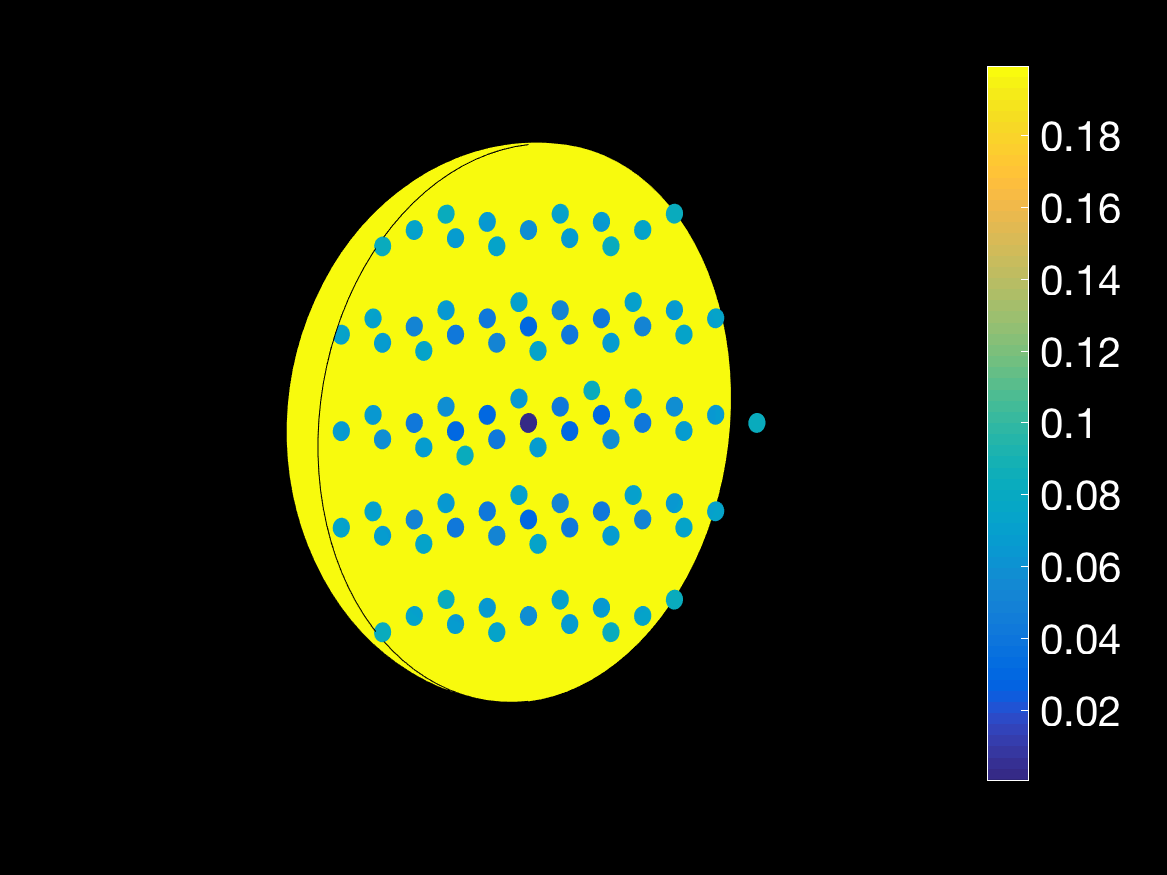}
\caption{The elastic deformation of  the embedded model when $R=3$. The colorcode represents the modulus of the displacement. 
 \label{illustration-2}}
\end{figure}

We report the result of the computational time in Figure~\ref{numberofshperes} which illustrates that the computational cost with respect to the number of spheres grows as $O(M^2)$. This is the normal scaling for an integral equation involving $M$ spheres, whose iterative solver requires a number of iterations that is independent of $M$ which we observe. 
 
\begin{figure}
\centering 
\includegraphics[trim=2in 3in 2in 2.8in,width=0.3\textwidth]{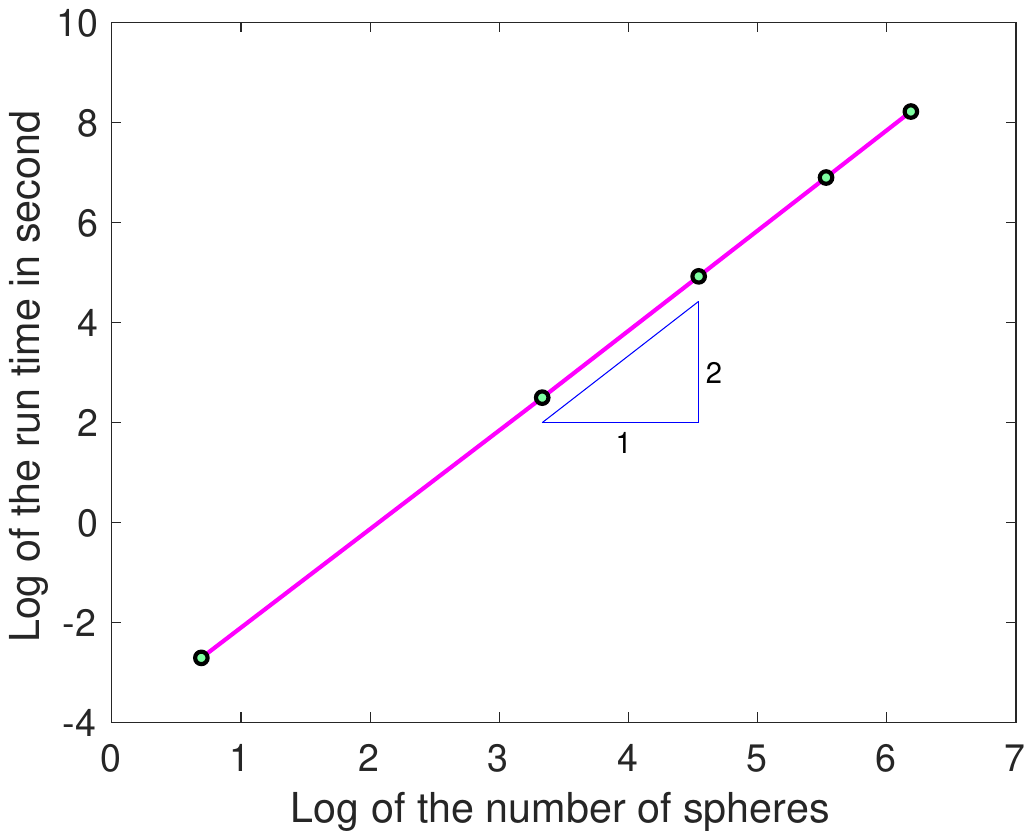}
\caption{The run time in second with respect to the number $M$ of spheres (log-log scale). 
 \label{numberofshperes}}
\end{figure}

\subsection{The effect of an inclusion}
We now consider the unit ball $B_1$ which contains an additional inclusion $\Omega_1$ in form of a sphere centered at the origin with radius $0.5$. 
We study how the displacement on the unit sphere $\mathbb S^2$ is influenced by the compressibility of the small inclusions $\Omega_1$.  We will use the Poisson's ratio as the parameter defined by  \eqref{Poisson-ratio} describing the compressibility of a substance.

Let the  stress tensor $- (x,y,z)^\top$ be imposed on the unit sphere $\mathbb S^2$ and fix the shear modulus  of the exterior shell $\Omega_0 = B_1\backslash\overline \Omega_1$ to be  $\mu_0=1$ and the shear modulus $\mu_1=1$ for the inclusion $\Omega_1$.  
We test several cases where the the exterior shell and the inclusion are associated to different Poisson's ratio  $\nu_0, \nu_1$. Recall that $\nu_0, \nu_1\in (-1,1/2)$ according to the definition, we have the limit values of the first Lam\'e parameter $\lambda_1$:
\[
\lambda_1\xrightarrow[\nu_1\to -1]{}-{2\over 3}, \quad \lambda_1\xrightarrow[\nu_1\to {1\over 2}]{}\infty. 
\]
In Figure~\ref{fig:poisson-ratio}, we plot the $L^2$ norm of the displacement on the unit sphere by letting the Poisson's ratio  $\nu_1$  vary in $[-1, 0.4998]$ with  different given values of Poisson's ratio $\nu_0$ of the background domain $\Omega_0$. In Figure~\ref{fig:illustration-poisson}, we give two solutions with different Poisson's ratios: the left solution  is obtained by setting $\nu_1=0.4995$ while the other is obtained by setting $\nu=-1$, both embedded into a background domain with $\nu_0=1/6$.

\begin{figure}
\centering
\includegraphics[trim= 0cm 7cm 0 7cm,clip, height = 2in]{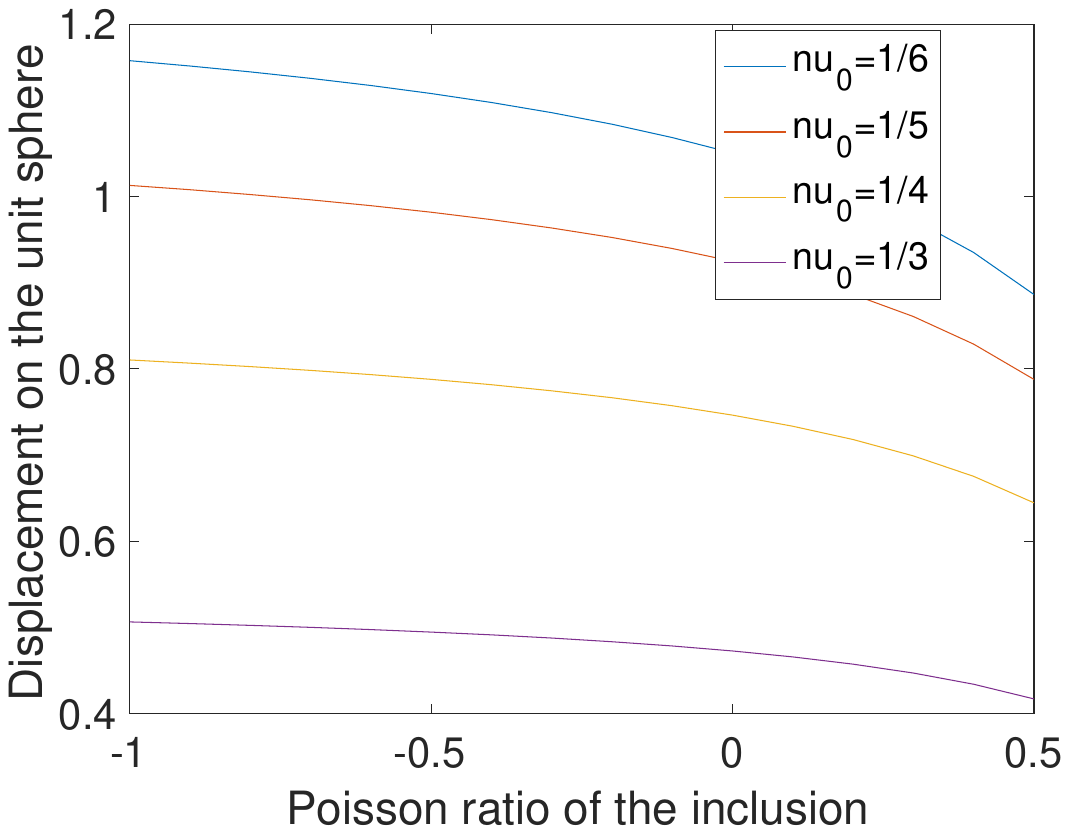} 
\caption{The  $L^2$ norm of the solution on the unit sphere with respect to the Poisson's ratio $\nu_1$ of the inclusion. Each curve is obtained by a given background Poisson's ratio $\nu_0$ specified by the legend. \label{fig:poisson-ratio}}
\end{figure}

\begin{figure}
\centering
\includegraphics[height = 2.5in]{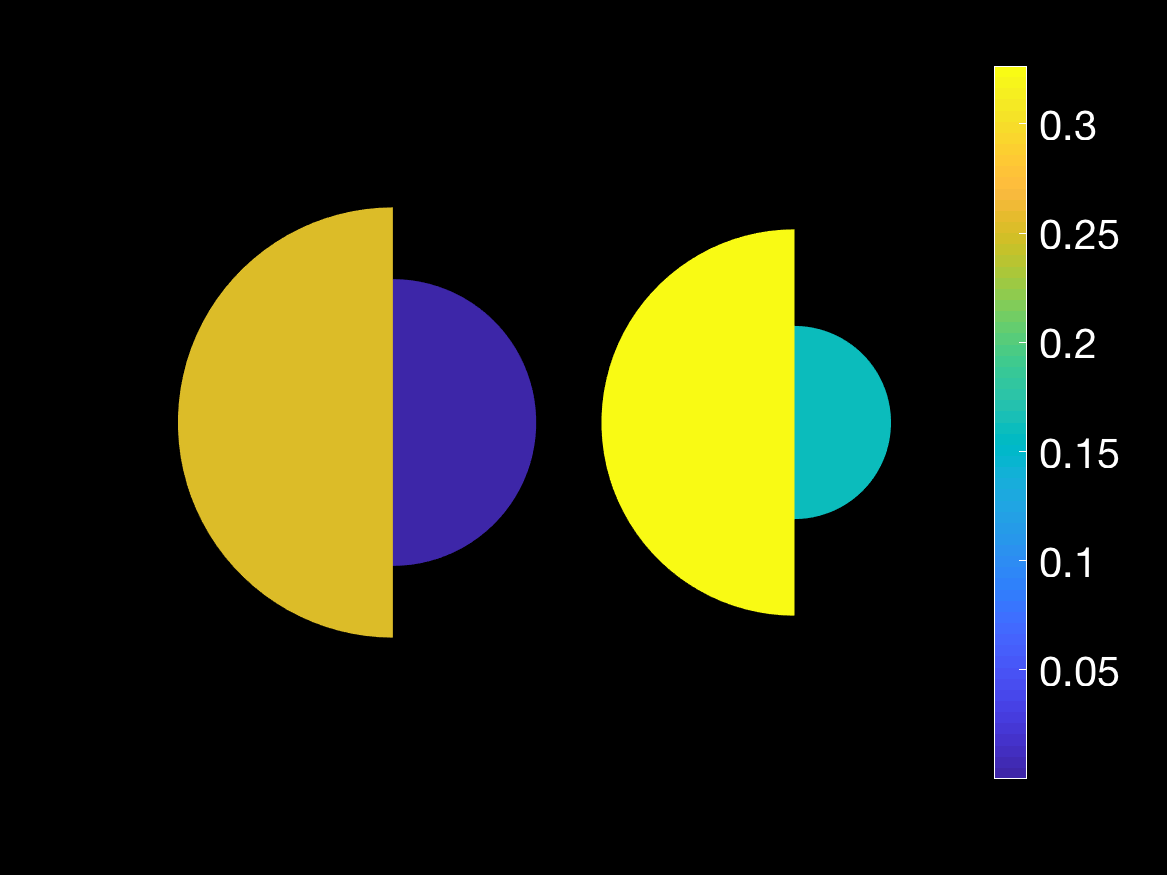} 
\caption{Two solutions with the Poisson's ratio  $\nu_0={1\over 6}$. The left solution is obtained for $\nu_1=0.4998$ while the left for $\nu_1=-1$. The colorcode represents the modulus of the displacement.\label{fig:illustration-poisson}}
\end{figure}

\section{Conclusion}
In this article, we have discussed the layer potentials and their corresponding integral operators on arbitrary bounded domains with Lipschitz boundary in the context of isotropic elasticity. 
We proved jump relations of layer potentials and the invertibility of the single layer boundary operator. In the particular case where the body is  a unit ball, we present spectral properties of the boundary operators on the base of the vector spherical harmonics. 
We then derived a second-kind integral equation for isotropic elastic materials with spherical inclusions that was then discretized by employing the vector spherical harmonics as basis functions and exploiting the spectral properties to enhance efficiency of the discretization. 
In the last part, we effect some numerical tests to asses the properties of the method: the accuracy with respect to the degree of the vector spherical harmonics and the complexity of the computational cost with respect to the number of spherical inclusions.  
We also used the method to explore how the deformation of the elastic material is effected by the value of the Poisson's ratio.


\section{Aknowledgement}
Benjamin Stamm acknowledges the funding from the German Academic Exchange Service (DAAD) from funds of the \^a Bundesministeriums f\"ur Bildung und Forschung\^a (BMBF) for the project Aa-Par-T (Project-ID 57317909). Shuyang Xiang acknowledges the funding from the PICS-CNRS as well as the PHC PROCOPE 2017 (Project N37855ZK).


\section*{Appendix A: computation of the first few vector spherical harmonics}

 We first start considering the table of vector spherical harmonics up to the second order as listed below:
 \begin{align*}
 	\ell=0: & \qquad Y_{0,0} = {1\over 2} \sqrt{1\over \pi},
	\\
 	\ell=1: & \qquad Y_{1,-1} = \sqrt {3\over 4\pi}y,  
 	& Y_{1,0} &= \sqrt {3\over 4\pi}z,  
 	& Y_{1,1} &= \sqrt {3\over 4\pi}x,
	\\
 	\ell=2: & \qquad Y_{2,-2}={1\over 2}  \sqrt{15 \over \pi}xy, 
 	& Y_{2,-1}&= {1\over 2}  \sqrt{15 \over \pi}yz,  
 	& Y_{2,0}& ={1\over 4} \sqrt {5\over \pi}(-x^2-y^2+2z^2), 
 	\\
 	& \phantom{\qquad } Y_{2,1} ={1\over 2}  \sqrt{15 \over \pi} xz,  
 	& Y_{2,2}&= {1\over 4}  \sqrt{15 \over \pi}(x^2-y^2). 
 \end{align*}
This gives first the obvious result that 
\[
	\ell=0: \qquad V_{0,0}= -{1\over 2} \sqrt{1\over \pi} (x,y,z)^\top \qquad\mbox{and}\qquad W_{00}=X_{00}=0.
\]
Using the definition of the surface gradient \eqref{surface}, we obtain 
\begin{align*}
\ell=1: &\qquad 
&\nabla_{\! \rm s} Y_{1,-1}&=\sqrt {3\over 4\pi} \big((0, 1,0)-y(x,y,z)\big)^\top,
\\
& \phantom{\qquad }
&\nabla_{\! \rm s} Y_{1,0}&= \sqrt {3\over 4\pi} \big((0, 0,1)-z(x,y,z)\big)^\top,
\\  
& \phantom{\qquad }
&\nabla_{\! \rm s} Y_{1,1}&= \sqrt {3\over 4\pi} \big((1,0,0)-x(x,y,z)\big)^\top,
\\
\ell=2: &\qquad 
&\nabla_{\! \rm s}  Y_{2,-2}&={1\over 2}  \sqrt{15 \over \pi}\big((y, x,0)-2xy(x,y,z)\big)^\top, 
\\
& \phantom{\qquad }
&\nabla_{\! \rm s} Y_{2,-1}&= {1\over 2}  \sqrt{15 \over \pi}\big((0,z,y)-2yz(x,y,z)\big)^\top,
\\
& \phantom{\qquad }
&\nabla_{\! \rm s} Y_{2,0}&={1\over 2} \sqrt {5\over \pi}\big((-x,-y,2z)-(-x^2-y^2+2z^2 )(x,y,z)\big)^\top,
\\
& \phantom{\qquad }
&\nabla_{\! \rm s} Y_{2,1}&={1\over 2}  \sqrt{15 \over \pi} \big((z,0,x)-2xz(x,y, z)\big)^\top, 
\\
& \phantom{\qquad }
&\nabla_{\! \rm s} Y_{2,2}&= {1\over 2 }  \sqrt{15 \over \pi} \big((x, -y,0)-(x^2-y^2)(x,y,z)\big)^\top. 
\end{align*}
The spherical harmonics  $V_{\ell m}$ up to  order 2 are then given as follows 
\begin{align*}
\ell=0: &\qquad 
&V_{0,0}&=  -{1\over 2} \sqrt{1\over \pi} (x,y,z)^\top,
\\
\ell=1: &\qquad 
&V_{1,-1}&=\sqrt {3\over 4\pi} \big((0, 1,0)-2y(x,y,z)\big)^\top, 
\\
& \phantom{\qquad }
&V_{1,0}&= \sqrt {3\over 4\pi} \big((0, 0,1)-2 z(x,y,z)\big)^\top,
\\
& \phantom{\qquad }
&V_{1,1}&= \sqrt {3\over 4\pi} \big((1,0,0)-2x(x,y,z)\big)^\top, 
\\
\ell=2: &\qquad 
&V_{2,-2}&={1\over 2}  \sqrt{15 \over \pi}\big((y, x,0)-5xy(x,y,z)\big)^\top,
\\
& \phantom{\qquad }
&V_{2,-1}&= {1\over 2}  \sqrt{15 \over \pi}\big((0,z,y)-5yz(x,y,z)\big)^\top,
\\
& \phantom{\qquad }
&V_{2,0}&={1\over 2} \sqrt {5\over \pi}\big((-x,-y,2z)-{5\over 2} (-x^2-y^2+2z^2 )(x,y,z)\big)^\top, ~ 
\\
& \phantom{\qquad }
&V_{2,1}&={1\over 2}  \sqrt{15 \over \pi} \big((z,0,x)-5xz(x,y, z)\big)^\top, 
\\
& \phantom{\qquad }
&V_{2,2}&= {1\over 2 }  \sqrt{15 \over \pi} \big((x, -y,0)-{5\over 2}  (x^2-y^2)(x,y,z)\big)^\top. 
\end{align*}
The spherical harmonics  $W_{\ell m}$ up to  order 2 are given by 
\begin{align*}
\ell=0: &\qquad 
& W_{0,0}&=0, 
\\
\ell=1: &\qquad 
& W_{1,-1}&=\sqrt {3\over 4\pi} (0, 1,0)^\top,
&W_{1,0}&= \sqrt {3\over 4\pi}(0, 0,1)^\top,
&W_{1,1}&= \sqrt {3\over 4\pi} (1,0,0)^\top,
\\ 
\ell=2: &\qquad 
& W_{2,-2} &={1\over 2}  \sqrt{15 \over \pi}(y, x,0)^\top,
&W_{2,-1}&= {1\over 2}  \sqrt{15 \over \pi}(0,z,y)^\top,
&W_{2,0}&={1\over 2} \sqrt {5\over \pi}(-x,-y,2z)^\top, 
\\  
& \phantom{\qquad }
& W_{2,1} &={1\over 2}  \sqrt{15 \over \pi} (z,0,x)^\top, 
&W_{2,2}&= {1\over 2 }  \sqrt{15 \over \pi} (x, -y,0)^\top. 
\end{align*}
And finally, the spherical harmonics $X_{\ell m}$ up to order 2 are given by 
\begin{align*}
\ell=0: &\qquad 
&X_{0,0}&=0,
\\
\ell=1: &\qquad 
&X_{1,-1}&=\sqrt {3\over 4\pi} (-z, 0,x)^\top,
&X_{1,0}&= \sqrt {3\over 4\pi}(y, -x,0)^\top,
\\  
& \phantom{\qquad }
&X_{1,1}&= \sqrt {3\over 4\pi} (0,z,-y)^\top,
\\ 
\ell=2: &\qquad 
&X_{2,-2}&={1\over 2}  \sqrt{15 \over \pi}(-xz, yz,x^2-y^2)^\top,
&X_{2,-1}&= {1\over 2}  \sqrt{15 \over \pi}(y^2-z^2,-xy,xz)^\top,
\\  
& \phantom{\qquad }
&X_{2,0}&={1\over 2} \sqrt {5\over \pi}(3yz,-3xz,0)^\top, 
&X_{2,1}&={1\over 2}  \sqrt{15 \over \pi} (xy, z^2-x^2,-yz)^\top, 
\\  
& \phantom{\qquad }
&X_{2,2}&= {1\over 2 }  \sqrt{15 \over \pi} (yz, xz,-2xy)^\top. 
\end{align*}
\section*{Appendix B: Entries of matrices $A_{\DL,\ell}^{in}$ and $A_{\DL,\ell}^{out}$}
The coefficients in $A_{\DL,\ell}^{in}$ and $A_{\DL,\ell}^{out}$ are given as follows: 
\label{App:B}
\begin{align*}
a^{in,\DL,\ell}_{11} &= - { (\ell+2)\big((3\ell+2)\mu+(\ell+1)\lambda\big)\big((3\ell+1)\mu+\ell\lambda\big)\over  ( 2\ell+3)(2\ell+1) ^2\mu (2\mu+\lambda)}, 
\\
a^{in,\DL,\ell}_{21,1} &= -{(\ell+1)(\ell+2)\big((3\ell+2)\mu+(\ell+1)\lambda\big)(\mu+\lambda)\over 2(2\ell+1)^2 (2\mu+ \lambda )},
\\
a^{in,\DL,\ell}_{21,2}  &= {(\ell+1)(\ell+2)\big((3\ell+2)\mu+(\ell+1)\lambda\big)(\mu+\lambda)\over 2(2\ell-1)(2\ell+1)\mu (2\mu+ \lambda )}, 
\\
a^{in,\DL,\ell}_{12} &= -{\ell(\ell-1)(\mu+\lambda)\big((3\ell+1)\mu+\ell\lambda\big)\over (2\ell+3)(2\ell+1)^2\mu (2\mu+\lambda)}, 
\\
a^{in,\DL,\ell}_{22,1} &= -{\ell(\ell-1)(\ell+1)(\mu+\lambda)^2\over2(2\ell+1)^2\mu (2\mu+\lambda)}, 
\\
a^{in,\DL,\ell}_{22,2} &=  {(\ell^3 +24\ell^2-5\ell-8)\mu^2 +2( \ell^3+6\ell^2-2\ell-2)\mu\lambda+ (\ell^3-\ell)\lambda^2\over (2\ell-1)(2\ell+1)\mu (2\mu+\lambda)},
\end{align*}
and 
\begin{align*}
a^{out,\DL,l}_{11,1}&={(\ell+1)\big((\ell^2+10\ell+4)\mu^2+(2\ell^2+8\ell+2)\mu\lambda+(\ell^2+\ell)\lambda\big)\over2(2\ell+1)(2\ell+3) \mu (2\mu+\lambda)},
\\
a^{out,\DL,\ell}_{11,2}&=-{\ell(\ell+1)(\ell+2)(\mu+\lambda) ^2\over 2(2\ell+1)^2\mu(2\mu+\lambda)}, 
\\
a^{out,\DL,\ell}_{21}&={(\ell+1)(\ell+2)(\mu+\lambda)\big((3\ell+2)\mu+(\ell+1)\lambda\big)\over (2\ell-1)(2\ell+1)^2\mu (2\mu+\lambda)}, 
\\
a^{out,\DL,\ell}_{12,1}&= -{ \ell(\ell-1)(\mu+\lambda)\big((3\ell+1)\mu+l\lambda\big)\over2  (2\ell+3)(2\ell+1)\mu (2\mu+\lambda^2)},
\\
a^{out,\DL,\ell}_{12,2}&={ \ell(\ell -1)\big((3\ell+1)\mu+\ell\lambda\big)(\mu+\lambda)\over 2 (2\ell+1)(2\ell+3)\mu (2\mu+\lambda)}, 
\\
a^{out,\DL,\ell}_{22}&={ (\ell -1)\big((3\ell+1)\mu+\ell\lambda\big)\big((3\ell+2)\mu+(\ell+1)\lambda\big)\over (2\ell-1)(2\ell+1)^2\mu (2\mu+\lambda)}. 
\end{align*}

\bibliographystyle{plain}
\bibliography{elasticity}

\begin{thebibliography}{10}

\bibitem{allan-bower}
Bower. A.
\newblock {Lecture notes: EN224: Linear Elasticity.}
\newblock Division of Engineering, Brown University, 2005.

\bibitem{BEG}
R.G. Barrera, G.A. Est{\'e}vez, and J.~Giraldo.
\newblock Vector spherical harmonics and their application to magnetostatics.
\newblock {\em Eur. J. Phys.}, 6(287-294), 1985.

\bibitem{BUI}
H.~D. Bui.
\newblock An integral eqautions method for solving the problem of a plane crack
  of arbitary shape.
\newblock {\em J. Mech. Phys. Solids}, 25:29--39, 1997.

\bibitem{PART1}
E.~Canc\`es, V.~Ehrlacher, F.~Legoll, B.~Stamm, and S.~Xiang.
\newblock {An embedded corrector problem for homogenization. Part I: Theory}.
\newblock {\em to appear in SIAM MMS}, 2020.

\bibitem{PART2}
Eric Canc{\`e}s, Virginie Ehrlacher, Fr{\'e}d{\'e}ric Legoll, Benjamin Stamm,
  and Shuyang Xiang.
\newblock {An embedded corrector problem for homogenization. Part II:
  Algorithms and discretization}.
\newblock {\em Journal of Computational Physics}, page 109254, 2020.

\bibitem{EELL}
B.~Carrascal, P.~Estevez, Lee, and V.~Lorenzo.
\newblock Vector spherical harmonics and their application to classical
  electrodynamics.
\newblock {\em Eur. J. Phys.}, 12(184-191), 1991.

\bibitem{integral}
Haxton D.J.
\newblock Lebedev discrete variable representation.
\newblock {\em Journal of Physics B: Atomic, Molec- ular and Optical Physics},
  (40):23, 2007.

\bibitem{Corona-Veera}
Corona E. and Veerapaneni S.
\newblock {Boundary integral equation analysis for suspension of spheres in
  Stokes flow}.
\newblock {\em Journal of Computational Physics}, 362:327--345, 2018.

\bibitem{Hill}
E.L.Hill.
\newblock The theory of vector spherical harmonics.
\newblock {\em Am. J. Phys}, 22(211- 214), 1954.

\bibitem{weinberg}
Weinberg.~E. J.
\newblock Monopole vector spherical harmonics.
\newblock {\em Phys. Rev. D}, 49:1086--1092, 1994.

\bibitem{JKO}
V.V. Jikov, S.M. Kozlov, and O.A. Oleinik.
\newblock {\em Homogenization of differential operators and integral
  functionals}.
\newblock Springer, Berlin, 1994.

\bibitem{P-S}
Phani~K. K. and Sanyal D.
\newblock {The relations between the shear modulus, the bulk modulus and
  Young's modulus for porous isotropic ceramic material}.
\newblock {\em Mater. Sci. Eng. A.}, 490(1):305--312, 2008.

\bibitem{KUPR}
V~D Kupradze.
\newblock {\em Progress in solid mechanics / Dynamical problems in
  elasticity.}, volume~3.
\newblock Amsterdam : North-Holland Publishing, 1963.

\bibitem{ManyBodyPolTheory}
Eric~B. Lindgren, Anthony~J. Stace, Etienne Polack, Yvon Maday, Benjamin Stamm,
  and Elena Besley.
\newblock {An integral equation approach to calculate electrostatic
  interactions in many-body dielectric systems}.
\newblock {\em Journal of Computational Physics}, 371:712--731, 2018.

\bibitem{Leroy}
Leroy.~Y. M.
\newblock Introduction to the finite-element method for elastic and
  elasto-plastic solids.
\newblock In {\em Mechanics of Crustal Rocks}, pages 157--239. Springer,
  Vienna, 2011.

\bibitem{MacRobert}
T.M. MacRobert.
\newblock {\em Spherical harmonics: an elementary treatise on harmonic
  functions, with applications}.
\newblock Pergamon Press, 1967.

\bibitem{McLean}
William McLean.
\newblock {\em Strongly elliptic systems and boundary integral equations}.
\newblock Cambridge university press, 2000.

\bibitem{PHCM}
Mott P.H. and Roland C.M.
\newblock {Limits to Poisson's ratio in isotropic materials}.
\newblock {\em Phys. Rev. B}, 80(132104), 2009.

\bibitem{falk}
Falk.~R. S.
\newblock Lecture notes: Finite element method for linear elasticity.
\newblock Department of Mathematics - Hill CenterRutgers, The State University
  of New Jersey, 2008.

\bibitem{sauter2010boundary}
Stefan~A Sauter and Christoph Schwab.
\newblock Boundary element methods.
\newblock In {\em Boundary Element Methods}, pages 183--287. Springer, 2010.

\bibitem{steinbach2007numerical}
Olaf Steinbach.
\newblock {\em Numerical approximation methods for elliptic boundary value
  problems: finite and boundary elements}.
\newblock Springer Science \& Business Media, 2007.

\bibitem{Urklain}
Kushch V.I.
\newblock {\em Effective Properties of Heterogeneous Materials}, chapter~2,
  pages 97--197.
\newblock Springer, February 2013.

\bibitem{hobson}
Hobson~E. W.
\newblock {\em The theory of spherical and ellipsoidal harmonics}.
\newblock Chelsea Pub. Co., 1955.

\end{thebibliography}
\end{document}